\documentclass[12pt, reqno]{amsart}
\usepackage{amsmath, amstext, amsbsy, amssymb, amscd, mathrsfs}
\usepackage{amsmath}
\usepackage{amsxtra}
\usepackage{amscd}
\usepackage{amsthm}
\usepackage{amsfonts}
\usepackage{amssymb}
\usepackage{eucal}
\usepackage{color}
\usepackage[all]{xy}

\newtheorem{theorem}{Theorem}

\newtheorem{lemma}[theorem]{Lemma}
\newtheorem{proposition}[theorem]{Proposition}
\newtheorem{corollary}[theorem]{Corollary}
\theoremstyle{definition}
\newtheorem{definition}[theorem]{Definition}
\newtheorem{example}[theorem]{Example}
\newtheorem{xca}[theorem]{Exercise}
\theoremstyle{remark}
\newtheorem{remark}[theorem]{Remark}

\newtheorem{problem}[theorem]{Problem}

\definecolor{A}{rgb}{.75,1,.75}

\newcommand{\ad}{\text{ad}\,}
\newcommand{\Ad}{\text{Ad}\,}
\newcommand{\brst}{\mathcal{B}^\bullet}
\renewcommand{\c}{\mathfrak c}
\newcommand{\C}{ \mathbb C}
\newcommand{\col}{ \text{col}}
\newcommand{\End}{\text{End}}
\newcommand{\ev}[1]{{#1}_{\bar{0}}}
\newcommand{\od}[1]{{#1}_{\bar{1}}}

\newcommand{\ga}{\mathfrak{g}}
\newcommand{\Ga}{\Gamma}
\newcommand{\gl}{\mathfrak{gl}}
\newcommand{\gr}{\text{gr} }
\newcommand{\Hom}{\text{Hom} }
\newcommand{\hf}{\frac12}
\newcommand{\lag}{\mathfrak{l}}
\newcommand{\la}{\lambda}

\newcommand{\m}{\mathfrak m}
\renewcommand{\a}{\mathfrak{m}}
\newcommand{\n}{\m'}
\newcommand{\N}{\mathbb N}
\newcommand{\op}{\text{op}}
\newcommand{\pth}[1]{{#1}^{[p]}} 
\newcommand{\pr}{\text{pr}}
\newcommand{\qn}{\mathfrak{q}_n}
\newcommand{\R}{ \mathbb R}
\newcommand{\sgl}{\mathfrak{sl}}
\newcommand{\s}{ \mathfrak{s} }
\newcommand{\sudim}{\underline{\text{dim}}\,}
\renewcommand{\t}{\mathfrak h}
\newcommand{\W}{\mc{W}_\chi}
\newcommand{\We}{\mc{W}_e}
\newcommand{\wh}{\text{Wh}}
\newcommand{\whmod}{\ga\texttt{-Wmod}^\chi}
\newcommand{\Wmod}{ \mc{W}_\chi\texttt{-mod}}
\newcommand{\Z}{ \mathbb Z }
\newcommand{\Zp}{\mathcal{Z}_p}

\newcommand{\mc}{\mathcal}
\newcommand{\mf}{\mathfrak}

{\vskip-\lastskip\medskip
  \noindent
  {\em #1.}\enspace
  }%
{\qed\par\medskip
  }

\begin{document}
\title[Finite $W$-algebras]
{Nilpotent orbits and finite $W$-algebras}

\author[Weiqiang Wang]{Weiqiang Wang}
\address{Department of Mathematics, University of Virginia,
Charlottesville, VA 22904} \email{ww9c@virginia.edu}
\thanks{Partially supported by NSF and NSA grants.}

\subjclass[2000]{Primary 17B37}

\begin{abstract}
In recent years, the finite $W$-algebras associated to a
semisimple Lie algebra $\ga$ and a nilpotent element of $\ga$ have
been studied intensively from different viewpoints. In this
lecture series, we shall present some basic constructions,
connections, and applications of finite $W$-algebras. The topics
include:

\begin{itemize}
\item
 various equivalent definitions of $W$-algebras
\item independence of $W$-algebras on choices of Lagrangian
subspaces

\item $W$-algebras as quantization of Slodowy slices

\item classification of good $\Z$-gradings in type $A$

\item independence of $W$-algebras on choices of good gradings

 \item category equivalence between $W$-algebra modules and Whittaker
$\ga$-modules

 \item higher level Schur duality between $W$-algebras and
 Hecke algebras
 \item $W$-(super)algebras in prime characteristic
\end{itemize}
\end{abstract}

\maketitle

\date{3/27/06}
  \setcounter{tocdepth}{1}
\tableofcontents

\section{Introduction}

\subsection{}  \label{W:subsec:prehist}

The finite $W$-algebras\index{$W$-algebra} are certain associative
algebras associated to a complex semisimple Lie algebra $\ga$ and
a nilpotent element $e \in \ga$. In the full generality, they were
introduced by Premet \cite{Pr1, Pr2} into mathematics in a
different terminology. Some important special cases go back to
Lynch's thesis \cite{Ly}, which is in turn a generalization of
Kostant's construction in the case when $e$ is regular nilpotent.
For $e=0$, a finite $W$-algebra is simply the universal enveloping
algebra $U(\ga)$. On the other extreme, the finite $W$-algebra for
regular nilpotent $e$ is isomorphic to the center of $U(\ga)$
\cite{Ko}. In the literature of mathematical physics, the finite
$W$-algebras appeared in the work of de Boer and Tjin \cite{BT}
from the viewpoint of BRST quantum hamiltonian reduction. The history is further complicated as there is an enormous amount of work on affine $W$-algebras in the 1990's preceding the recent interest of finite $W$-algebras (and we apologize for inadequate references in this regard). See the book of Bouwknegt and Schoutens \cite{BS} and the vast references therein; also see some more precise comments below.

In the last few years, the finite $W$-algebras and their representations
have been extensively studied in mathematics with various new connections
developed.
The goal of the lecture series is to provide an introduction to
some basic aspects of the
fast-growing area of finite $W$-algebras. We sometimes sketch
the main ideas of proofs when it is not practical to reproduce
complete proofs.
These lecture notes are purely expository in nature and contain no
original results.

\subsection{}
\label{W:subsec:history}

Premet \cite{Pr1} started out with a characteristic $p$ version of
the finite $W$-algebras, and used it to settle a long-standing
conjecture of Kac and Weisfeiler \cite{WK} (cf. the review of
Jantzen \cite{Ja1} on modular representations of Lie algebras).
Subsequently, Premet \cite{Pr2} made a transition to
characteristic zero and showed that, with respect to the so-called
Kazhdan filtration\index{Kazhdan filtration} on the finite
$W$-algebras (cf. \cite{Ko, Mo}), the associated graded algebra is
the algebra of functions on the Slodowy slice through $e$
\cite{Slo}. In Premet's definition, one starts with a $\Z$-grading
$\Gamma: \ga =\oplus_{j \in \Z} \ga_j$ which is induced by a
standard $\mf{sl}_2$-triple $\{e, h, f\}$ (nowadays such a
$\Z$-grading is referred to as a Dynkin $\Z$-grading). The
definition of finite $W$-algebras $\mc W_{e, \Gamma,\lag}$ depends
on the choice of a Dynkin $\Z$-grading $\Gamma$ and also the
choice of a Lagrangian subspace $\lag$ of $\ga_{-1}$ (which is
known to be a symplectic vector space).

Gan and Ginzburg \cite{GG} developed a purely characteristic zero
treatment in a generalized form of Premet's theorem above which
allows us to regard finite $W$-algebras as quantizations of the
Slodowy slices. Moreover, Gan and Ginzburg modified Premet's
definition by starting with an arbitrary {\em isotropic} (instead
of Lagrangian) subspace $\lag$ of $\ga_{-1}$ and showed that
different choices of isotropic (in particular, Lagrangian)
subspaces $\lag$ gave rise to isomorphic finite $W$-algebras $\mc
W_{e,\Gamma}$. The approach of Gan and Ginzburg also provides a
new and more transparent proof of Skryabin's
equivalence\index{Skryabin equivalence} of categories \cite{Sk}
between a module category of $\mc W_{e,\Gamma}$ and a category of
Whittaker modules corresponding to $e$. The Whittaker $\ga$-module
was earlier studied in Kostant's paper \cite{Ko} when $e$ is
regular nilpotent.

Various key ingredients for finite $W$-algebras actually appeared earlier in the 1990's in the setting of  (classical and quantum) affine $W$-algebras. In particular, Feh\'er {\em et. al.} \cite[Section~III]{FRTW1} formulated the classical affine $W$-algebras in the generality of good $\Z$-gradings of $\ga$ (as a natural
generalization of the Dynkin $\Z$-gradings) for an arbitrary
nilpotent element $e \in \ga$, though the terminology of good $\Z$-gradings was formalized later by Kac, Roan and Wakimoto \cite{KRW}.
Feh\'er {\em et. al.} \cite[Section~3.4]{FRTW2} also recognized the key role played by a Lagrangian subspace of $\ga_{-1}$ in classical affine $W$-algebras. The quantum affine $W$-algebras associated to a regular nilpotent $e$ from the BRST approach (also called quantized Drinfeld-Sokolov reduction in the
affine setup) were studied from the viewpoint of vertex algebras by Frenkel, Kac, and Wakimoto \cite{FKW}. The finite and affine $W$-algebras
are related via a fundamental construction of Zhu algebra in the
theory of vertex algebras \cite{Zhu} (cf. Arakawa \cite{Ar}, De Sole and Kac
\cite{DK}). The good $\Z$-gradings were subsequently classified by
Elashvili and Kac \cite{EK}; also see Baur and Wallach \cite{BW}
for closely related work and further clarification.

Brundan and Goodwin \cite{BG} showed that different good
$\Z$-gradings $\Gamma$ on $\ga$ for $e$ led to isomorphic finite
$W$-algebras $\We$ (where we drop the dependence on $\Gamma$).
Actually, they worked and established the results in a somewhat
more general and flexible setting of good $\R$-gradings, which
allowed them to introduce some suitable geometric object called
good grading polytopes.

It is clear from the definition that two nilpotent elements in
$\ga$ which are conjugate under the adjoint group $G$ lead to
isomorphic finite $W$-algebras. Hence, the isomorphism classes of
finite $W$-algebras ultimately depend only on the nilpotent orbits
in $\ga$.

The representation theory of finite $W$-algebras has been
developed most adequately in type $A$ in a series of papers by
Brundan and Kleshchev \cite{BK1, BK2,BK3}. In particular, there
exists a remarkable higher level duality between finite
$W$-algebras and cyclotomic Hecke algebras, which recovers the
classical Schur duality at level one \cite{BK3}.

The constructions of finite $W$-algebras afford a natural
superalgebra generalization. In characteristic $p$, they are
developed in Lei Zhao's Virginia dissertation (see \cite{WZ1, WZ2,
Zh}), in connection with the formulation and proof of a super
analogue of the Kac-Weisfeiler conjecture when $\ga$ is a basic
classical Lie superalgebra or a queer Lie superalgebra. Many
aspects of the representation theory of finite $W$-algebras over
$\C$ afford a superalgebra generalization, and they are currently
being investigated.

\subsection{}
\label{W:subsec:outhistory}

For lack of expertise and space, we have left out many interesting
topics. We refer to Section~\ref{W:sec:open} for somewhat more
detailed comments on these topics and discussions of some open
questions. We also refer to a more recent survey by Losev
\cite{Lo5} which covers various complementary topics in detail.

Here is a quick outline on various topics which are (regretfully)
{\em not} to be covered in details in these lecture notes. Losev
\cite{Lo1} has developed a powerful deformation quantization
approach toward finite $W$-algebras. Connections between shifted
Yangians and finite $W$-algebras of type $A$ \cite{BK1} (also cf.
\cite{RS} in some special cases) allow Brundan and Kleshchev
\cite{BK2} to classify the finite-dimensional irreducible
$\We$-modules and obtain Kazhdan-Lusztig type character formula in
type $A$ (see \cite{VD} for formulation of Kazhdan-Lusztig
conjecture for finite $W$-algebras). Deep connections between
finite-dimensional modules of finite $W$-algebras and primitive
ideals have been developed \cite{Pr3, Pr4, Pr5, Gi, Lo1, Lo2, Lo4,
GRU}, and there has also been interesting work on category
$\mathcal O$ of finite $W$-algebras, see \cite{Pr3, BGK, Lo3, Go,
We}.

\subsection{}
\label{W:subsec:plan}

 The plan of the lecture notes is as follows.

In Section~\ref{W:sec:orbit}, a bijection between $\sgl_2$-triples
modulo $G$-conjugation and nonzero nilpotent orbits in $\ga$ is
formulated. We introduce the Dynkin and good $\Z$-gradings, and
then explain some of their basic properties.

In Section~\ref{W:sec:defW}, the finite $W$-algebras are defined
in several different ways and these definitions are shown to be
equivalent.

In Section~\ref{W:sec:Slodowy}, we formulate the finite
$W$-algebras based on an isotropic subspace $\lag$ of $\ga_{-1}$.
Then we introduce the Slodowy slices and the Kazhdan filtration on
finite $W$-algebras, with respect to which the associated graded
algebra is isomorphic to the algebra of functions on Slodowy
slices. The proof of the independence of the choices of $\lag$ for
the isoclasses of finite $W$-algebras is reduced to some Lie
algebra cohomology calculation, which is completed in
Section~\ref{W:sec:cat=}. The same type of argument is used to
derive Skryabin's equivalence in Section~\ref{W:sec:cat=}.

In Section~\ref{W:sec:A}, the classification of good
$\Z$-gradings in type $A$ is presented, in terms of some combinatorial
objects called pyramids.

In Section~\ref{W:sec:indep good}, different good
$\Z$-gradings are shown to lead to isomorphic finite $W$-algebras.

In Section~\ref{W:sec:duality}, we formulate and outline the
higher level Schur duality between finite $W$-algebras of type $A$
and cyclotomic Hecke algebras.

In Section~\ref{W:sec:charp}, we formulate $W$-superalgebras in
characteristic $p$, and its role in proof of super Kac-Weisfeiler
conjecture. This includes the Lie algebra counterpart as its most
important special cases.

In Section~\ref{W:sec:open}, we briefly outline some major topics
on finite $W$-algebras which are not covered in these lecture
notes, and then list some open problems.

\section{Nilpotent orbits, Dynkin and good $\Z$-gradings}
\label{W:sec:orbit}
\subsection{Basic setups for Lie algebras}
\label{W:subsec:setup}

For a Lie algebra $\mf a$, we will denote by $U(\mf a)$ the
universal enveloping algebra of $\mf a$. Let $\mf a_e$ denote the
centralizer of $e \in  \mf a$. Sometimes it is convenient to
regard $\mf a_e =\ker (\ad e)$, the kernel of the adjoint operator
$\ad e: \mf a \rightarrow \mf a$.

Let $\ga$ be a finite dimensional semisimple or reductive Lie
algebra over $\C$ equipped with a non-degenerate invariant
symmetric bilinear form $(\cdot|\cdot)$. Throughout the lectures,
we almost always take $\ga$ either to be a simple Lie algebra or
the general linear Lie algebra $\gl_N$. For $\ga =\gl_N$, we may
take the bilinear form by letting $(a|b) =\text{tr} (ab)$ where
$ab$ denotes the matrix multiplication of $a,b \in \gl_N$. An
element $e \in \mathfrak{g}$ is {\em nilpotent} means that $\ad e$
is a nilpotent endomorphism on $\ga$. For example,  a nilpotent
element in $\gl_N$ is exactly an $N\times N$ matrix with all
eigenvalues $0$.

A {\em $\Z$-grading} of Lie algebra $\ga$ is a decomposition $\ga
= \oplus_{j \in \Z} \mathfrak{g}_j$ which is always assumed to
satisfy the graded commutation relation:  $[\ga_i, \ga_j]\subseteq
\ga_{i+j}, \forall i,j.$

\subsection{Nilpotent orbits}
\label{W:subsec:NilOrb}

Denote by $G$ the adjoint group associated to the Lie algebra
$\ga$, and by $\mathcal O_e$ the $G$-adjoint orbit of $e$ in
$\ga$. We refer to Collingwood and McGovern \cite{CMc} for more on
nilpotent orbits.

Denote by $\mathcal N$ the null cone which consists of all
nilpotent elements in $\ga$. There exists a unique dense open
orbit $\mathcal O_{reg}$ in $\mathcal N$, called the {\em regular
nilpotent orbit}. A nilpotent element $e \in \ga$ is {\em regular
nilpotent} if and only if $\ga_e$ attains the minimal dimension
(which equals the rank of $\ga$). The regular nilpotent orbit
$\mathcal O_{reg}$ consists of all the regular nilpotent elements
in $\ga$. There is a unique dense open orbit in $\mathcal N
\backslash \mathcal O_{reg}$, called the {\em subregular orbit}
and denoted by $\mathcal O_{sub}$. There exists a unique {\em
minimal orbit} $\mathcal O_{min}$ in $\mathcal N$ of smallest
positive dimension. The dimension of $\mathcal O_{min}$ is known
\cite{Wa} to be $2h^\vee -2$, where $h^\vee$ is the dual Coxeter
number of $\ga$. Incidentally, the dual Coxeter number plays a key
role in representation theory of affine Lie algebras.

The set of nilpotent orbits in $\ga$ is naturally a poset
$\mathcal P$ with partial order $\le$ as follows: $\mathcal O' \le
\mathcal O$ if and only only if $\mathcal O' \subseteq
\overline{\mathcal O}$ (the closure of $\mc O$). Then $\mathcal
O_{reg}$ is the maximum and the zero orbit is the minimum in the
poset $\mathcal P$. Moreover, $\mathcal O_{sub}$ is the maximum in
the poset $\mathcal P \backslash \{\mathcal O_{reg}\}$, while
$\mathcal O_{min}$  is the minimum in the poset $\mathcal P
\backslash \{0\}$.

\begin{example}  \label{W:ex:prin sl2}
Let $\ga =\gl_n$. Every nilpotent matrix in $\gl_n$ is conjugate
to a matrix in Jordan canonical form. It follows that the
nilpotent orbits in $\ga$ are parameterized by the partitions $\la
=(\la_1, \la_2,\ldots)$ of $n$, where the orbit $\mathcal O_\la$
contains the matrix in standard Jordan form $\text{diag}
(J_{\la_1}, J_{\la_2}, \ldots)$, and we have denoted by $J_k$ the
$k\times k$ Jordan block.

Associated to the four distinguished partitions $(n)$, $(n-1,1)$,
$(2,1^{n-2})$, $(1^n)$ are the orbits $\mathcal O_{reg}$,
$\mathcal O_{sub}$,  $\mathcal O_{min}$ and the zero orbit
$\{0\}$, respectively.
\end{example}

\subsection{Jacobson-Morozov Theorem} \label{W:JacobsonM}

Associated to a nonzero nilpotent element $e \in \ga$, there
always exists $\{e, h, f\} \subseteq \ga$ which satisfy the
$\sgl_2$ commutation relation:
$$
[e,f]=h, \quad [h,e]=2e, \quad  [h,f]=-2f.
$$ We will
refer to this as a {\em standard $\sgl_2$-triple} or simply an
{\em  $\sgl_2$-triple}.

This statement is known as the {\em Jacobson-Morozov Theorem}, and
it can be proved by induction on the dimension of $\ga$.
See Carter \cite[5.3]{Ca} or \cite[3.3]{CMc}.

\begin{example} \label{W:ex:regular}
Let $\ga =\gl_n$, $e =J_n$, $h =\text{diag}\,(n-1, n-3, \ldots,
3-n, 1-n)$, and let
$$
f =\begin{pmatrix} 0 &   &   & 0 \\
a_1 & 0 &   & \\
&  \ddots & \ddots & \\
0 & & a_{n-1}&0
\end{pmatrix}
$$
with exactly one nonzero (sub)diagonal with entries $a_i =i(n-i)$
for $1\le i\le n-1$.
Then $\{e,h,f\}$ form a standard $\sgl_2$-triple.
\end{example}

\subsection{The Dynkin gradings}
\label{W:Dynkin gr}

By the representation theory of $\sgl_2 =\C\{e, h, f\}$, the
eigenspace decomposition of the adjoint action $\ad h: \ga
\rightarrow \ga$ provides a $\Z$-grading
$$
\ga =\bigoplus_{j\in \Z} \ga_j,
$$
that is, $\ga_j =\{x\in \ga \mid \ad h (x) =jx\}.$
A $\Z$-grading of $\ga$ arising from $\sgl_2$ this way will be
referred to as a {\em Dynkin $\Z$-grading}. A Dynkin grading affords
the following favorable properties:
\begin{align}
 & e \in \ga_2,  \label{W:deg=2}\\
 \ad e: & \ga_j \rightarrow \ga_{j+2}
  \text{ is injective for } j \le -1,  \label{W:inj}\\
 \ad e: & \ga_j \rightarrow \ga_{j+2}
  \text{ is surjective for } j \ge -1, \label{W:surj}
\\
 \ga_e & \subseteq  \oplus_{j \ge 0} \ga_j,  \label{W:vanish}  \\
 (\ga_i| \ga_j) & =0 \text{ unless } i+j =0, \label{W:pairing} \\
 \dim \ga_e & = \dim \ga_0 +\dim \ga_1. \label{W:numbersimple}
\end{align}

Note that (\ref{W:deg=2}) follows from the $\sgl_2$-triple
definition. Claims \eqref{W:inj} and \eqref{W:surj} are immediate
consequences of the highest weight theory of $\sgl_2$ when applied
to the the adjoint module $\ga$. Claim~ (\ref{W:vanish}) follows
from  \eqref{W:inj} and \eqref{W:surj}. Formulas \eqref{W:pairing}
and \eqref{W:numbersimple} are also easy, and they will be proved
in a generalized setting in Proposition~ \ref{W:goodprop} below
using only  \eqref{W:deg=2}, \eqref{W:inj}, and \eqref{W:surj}.

For $k \in \Z$, we shall denote $\ga_{>k} =\oplus_{j>k}\ga_j$.
Similarly, we define $\ga_{\geq k}, \ga_{<k}, \ga_{\leq k}$.

\subsection{The good $\Z$-gradings} \label{W:sec:goodgrading}

In this subsection, we recall the notion of good $\Z$-gradings (see \cite[(3.3)]{FRTW1} and
\cite{KRW}) and its basic properties \cite{EK}.

\begin{definition}   \label{W:def:good}
A $\Z$-grading $\Ga: \ga =\oplus_{j\in \Z} \ga_j$ for $\ga$
semisimple is called a {\em good} $\Z$-grading for $e$ if it
satisfies the conditions \eqref{W:deg=2}-\eqref{W:surj} above.
\end{definition}

For $\ga$ reductive, a good $\Z$-grading $\ga =\oplus_{j\in \Z} \ga_j$
for $e$ is required to satisfy \eqref{W:deg=2}-\eqref{W:surj}
and an additional condition that the center $\mathfrak
z (\ga)$ of $\ga$ is contained in $\ga_0$.

In particular, $\ad e: \ga_{-1} \rightarrow \ga_1$ is always a
bijection.

The following simple lemma will be repeatedly used.
\begin{lemma} \label{W:grading}
For any $\Z$-grading $\Gamma: \ga =\oplus_{j\in \Z} \ga_j$, there
exists a semisimple element $h_\Gamma \in [\ga, \ga]$ such that
$\Gamma$ coincides with the eigenspace decomposition of $\ad
h_\Gamma$, i.e. $\ga_j =\{x \in \ga \mid [h_\Gamma, x] =jx\}$.
\end{lemma}
\begin{proof}
The degree operator $\partial: \ga \rightarrow \ga$ which sends $x
\mapsto jx$ for $x \in \ga_j$ is a derivation on the semisimple
Lie algebra $[\ga,\ga]$, hence an inner derivation on $[\ga,\ga]$,
given by $\ad h_\Gamma$ for some semisimple element $h_\Gamma \in
[\ga, \ga]$. Then $\partial =\ad h_\Gamma$ as derivations on $\ga
= [\ga,\ga] \oplus \mathfrak z (\ga)$ since the equality holds on
$\mathfrak z (\ga)$ too.
\end{proof}

\begin{proposition} \label{W:goodprop}
Properties (\ref{W:vanish})-(\ref{W:numbersimple}) remain to be
valid for every good $\Z$-grading $\Ga: \ga =\oplus_{j\in \Z}
\ga_j$.
\end{proposition}

\begin{proof}
Clearly $\ga_e$ is a $\Z$-graded Lie subalgebra of $\ga$, and
(\ref{W:vanish}) follows from (\ref{W:inj}).

For $x \in \ga_i, y \in \ga_j$, we have $-i (x|y) =([x,h_\Gamma] |
y) =(x | [h_\Gamma,y]) =j (x|y)$ by Lemma~\ref{W:grading}. This
implies (\ref{W:pairing}).

Finally, (\ref{W:numbersimple}) follows by an exact sequence of
vector spaces:
\begin{displaymath}
0 \longrightarrow \ga_e {\longrightarrow} \ga_{-1} +\ga_0
+\ga_{>0}
    \overset{\ad e}{\longrightarrow}
     \ga_{>0}  \longrightarrow 0
  \end{displaymath}
which is well-defined by \eqref{W:inj} and \eqref{W:surj}.
\end{proof}

\begin{example}  \label{W:ex:e=0}
Associated to $e=0$, we have a good $\Z$-grading with the whole
$\ga$ concentrated on degree zero, i.e. $\ga_0 =\ga$.
\end{example}

\begin{example} \label{W:ex:sl3}
Let $\ga =\gl_3$ and $e =E_{13}.$  

\begin{enumerate}
\item Let $h =\text{diag}\,(1, 0, -1)$ and $f=E_{31}$. Then
$\{e,h,f\}$ forms an $\sgl_2$-triple, and it gives rise to a
Dynkin $\Z$-grading of $\gl_3$ for $e \in \ga_2$ whose degrees on
the matrix units $E_{ij}$ are listed in the matrix below:
$$
\begin{pmatrix} 0 &   1   & 2 \\
-1 & 0 & 1 \\
-2 & -1 &0
\end{pmatrix}
$$
A basis of the  ($5$-dimensional) centralizer $\ga_e$ consists of
$2$ degree-zero elements $E_{11}+E_{33}, E_{22}$, $2$ degree-one
elements $E_{12},E_{23}$, and a degree-two element $E_{13}$.

\item The eigenspace decomposition of the diagonal element
$\text{diag}\ (1, 1, -1)$ gives rises to a good (but non-Dynkin)
$\Z$-grading for $e =E_{13}$ whose degrees on the matrix units
$E_{ij}$ are listed in the matrix below:
$$
\begin{pmatrix} 0 &   0   & 2 \\
0 & 0 & 2 \\
-2 & -2 &0
\end{pmatrix}
$$
Note now $\ga_{-1} =0$.
\end{enumerate}

In short, {\em Dynkin is good}, but {\em being good is not good
enough for being Dynkin.}
\end{example}

\begin{lemma} \label{W:inj=sur}
Properties \eqref{W:inj} and \eqref{W:surj} are equivalent for any
$\Z$-grading $\ga =\oplus_{j\in \Z} \ga_j$.
\end{lemma}
\begin{proof}
Note the proof of \eqref{W:pairing} in
Proposition~\ref{W:goodprop} only uses the $\Z$-grading of $\ga$,
but does not use \eqref{W:inj} and \eqref{W:surj}. Then we have a
non-degenerate pairing between $\ga_i$ and $\ga_{-i}$ induced from
the non-degenerate form $(\cdot|\cdot)$ on $\ga$. The lemma
follows by this non-degeneracy and the equation $(x_i | [e, y_j])
=([x_i,e] | y_j)$ for $i +j =-2$ and $i \le -1 \le j$.
\end{proof}

Let $\Gamma: \ga =\oplus_{j \in \Z} \ga_j$ be a good $\Z$-grading
of $\ga$. Note that $h_\Gamma$ defined in Lemma~\ref{W:grading}
actually lies in $\ga_0$. Then $\ga_0 =\ga_{h_\Gamma}$ is a Levi
(reductive) subalgebra of $\ga$. Choose $\mf h$ to be a Cartan
subalgebra of $\ga_0$, and hence also a Cartan subalgebra of
$\ga$. Let $\ga =\mf h \oplus \oplus_{\alpha \in \Phi} \ga^\alpha$
be a root space decomposition. Fix a system of positive roots
$\Delta_0^+$ of $(\ga_0, \mf h)$.

\begin{lemma}  \label{W:rootgraded}
Let $\Gamma: \ga =\oplus_{j \in \Z} \ga_j$ be a good $\Z$-grading
of $\ga$. Then,
\begin{enumerate}
\item each root subspace $\ga^\alpha$ lies in $\ga_j$ for some $j
\in \Z$;

\item $\Delta^+ := \Delta_0^+ \bigcup \{\alpha \mid \ga^\alpha
\subseteq \ga_{>0}\}$ is a system of positive roots of $(\ga, \mf
h)$.
\end{enumerate}
\end{lemma}

\begin{proof}
(1) Take $h_\Gamma \in \ga$ as in Lemma~\ref{W:grading}. Then
$[h_\Gamma, x] =\alpha (h_\Gamma) x$ for $x\in \ga^\alpha$ implies
that $\ga^\alpha \subseteq \ga_j$ for $j =\alpha (h_\Gamma) \in
\Z$.

(2) Choose a regular element $r\in \mf h$. Then $\epsilon r
+h_\Gamma$ for small $\epsilon >0$ is also regular, and so
$\Delta^+ =\{\alpha \in \Phi \mid \alpha (\epsilon r +
h_\Gamma)>0\}$ is a system of positive roots of $\ga$.
\end{proof}

Let $\Pi$ be the set of simple roots for $\Delta^+$, and denote
$\N =\{0,1,2,\ldots \}$. Then,
$$
\Pi = \bigcup_{j \in \N} \Pi_j, \qquad
 \text{where } \Pi_j =\{\alpha \in \Pi \mid \ga^\alpha \subseteq
 \ga_j\}.
$$
\begin{lemma} \label{W:weighted}
For a good $\Z$-grading $\Gamma$, we have $\Pi =\Pi_0 \cup \Pi_1
\cup \Pi_2$.
\end{lemma}
\begin{proof}
Assume to the contrary that there exists $\beta \in \Pi_s$ for
$s>2$. Since $e \in\ga_2$ and $\ga_2$ is contained in the
subalgebra generated by $\ga^\alpha$ with $\alpha \in \Pi_0 \cup
\Pi_1 \cup \Pi_2$, we have $[e, \ga^{-\beta}] =0$, or in other
words, $\ga^{-\beta} \subseteq \ga_e$.
This contradicts with the fact \eqref{W:vanish} that $\ga_e$ has
$\N$-grading as $\ga^{-\beta} \subseteq \ga_{-s}$.
\end{proof}

\subsection{A bijective map $\Omega$}
\label{W:bijection}

Clearly, the group $G$ acts on the collection of $\sgl_2$-triples
in $\ga$ by conjugation. Define a map
 \begin{align*}
  \Omega: \{ \text{$\sgl_2$-triples}\}/ G \longrightarrow &
 \{\text{nonzero nilpotent orbits}\}, \\
  \Omega (\{e,h, & f\}) =\mathcal O_e.
 \end{align*}

 \begin{theorem}  \label{W:th:omega}
The map $\Omega$ is a bijection.
 \end{theorem}

\begin{proof}[Sketch of a proof]
According to Jacobson-Morozov Theorem in \ref{W:JacobsonM}, the
map $\Omega$ is surjective. A theorem of Kostant asserts that the
map $\Omega$ is also injective, as we sketch below (see \cite[3.4]{CMc}
for details).

To that end, take two $\sgl_2$-triples which are mapped by
$\Omega$ to the same nilpotent orbit $\mathcal O_e$. Applying a
$G$-conjugation, we may assume without loss of generality that the
two triples are $\{e,h,f\}$ and $\{e,h',f'\}$, that is, they share
the same $e$.

Let $\mathfrak u_e := \ga_e \cap [\ga, e]$.
Note that $h'-h \in \mathfrak u_e$ and $\mathfrak
u_e$ is an $\ad h$-invariant, nilpotent ideal of $\ga_e$. Actually
$\mathfrak u_e = \oplus_{j>0} \ga_e \cap \ga_j$. Denote by $U_e$
the subgroup of $G$ with Lie algebra $\mf u_e$.

By using the fact that $\mf u_e$ is a $\N$-graded nilpotent ideal,
one shows easily that $\Ad U_e (h) =h +\mathfrak u_e$. Hence,
there exists $x \in U_e$ such that $\Ad x (h) =h'$.

Note that we automatically have $\Ad x (e) =e$ for $x\in U_e$. Now
the injectivity of $\Omega$ follows from  the {\em rigidity of
$\sgl_2$} below.

\begin{lemma}  \label{W:rigid}
Let $\{e,h,f\}$ and $\{e,h,f'\}$ be standard $\sgl_2$-triples of $\ga$.
Then $f=f'$.
\end{lemma}
\end{proof}

\begin{xca}   \label{W:xca:rigidity}
Prove Lemma~\ref{W:rigid}.
\end{xca}

\begin{remark}  \label{W:Molcev}
If one works a bit harder (cf. \cite[3.4]{CMc}), one recovers a theorem
of Molcev which asserts the injectivity of the map
 \begin{align*}
  \Theta: \{ \text{$\sgl_2$-triples}\}/ G  \longrightarrow &
 \{\text{nonzero semisimple orbits}\}, \\
  \Theta (\{e,h, & f\})  =\mathcal O_h.
 \end{align*}
\end{remark}

Given an $\sgl_2$-triple $\{e,h,f\}$, by Lemma~\ref{W:weighted}
(which applies to the special case of Dynkin grading $\ad h$) we
have a system of simple roots of $\ga$ whose weights on $h$ only
take values from $\{0,1,2\}$. Therefore the number of semisimple
orbits with such weight restriction is $\le 3^{\text{rank}(\ga)}$.
It now follows by the properties of $\Omega$ and $\Theta$ above
that the number of nilpotent orbits in $\ga$ is finite.

An alternative way of obtaining the finiteness of nilpotent orbits
of $\ga$ of classical type is to use the parametrization by
partitions (cf. \cite[Chapter~1]{Ja2}). But such an alternative ceases to
work well for Lie algebras $\ga$ of exceptional type.

\section{Definitions of $W$-algebras}
\label{W:sec:defW}
\subsection{The endomorphism algebra definition}

Assume we are given a reductive Lie algebra $\ga$ with a
non-degenerate invariant bilinear form $(\cdot|\cdot)$, an
nilpotent element $e \in \ga$, and a good $\Z$-grading $\Ga:
\ga=\oplus_{j\in \Z} \ga_j$ for $e$.

There exists a unique $\chi \in \ga^*$ such that $\chi (x) = (x |
e)$. Define a bilinear form $\langle\cdot,\cdot\rangle$ on
$\mathfrak{g}_{-1}$ as follows:
$$
\langle \cdot,\cdot\rangle: \ga_{-1} \times \ga_{-1} \rightarrow \C,
\qquad \langle x,y \rangle := ([x,y] |e) =\chi([x,y]).
$$
The following lemma was noticed earlier \cite[3.4]{FRTW2} in the setting for classical affine $W$-algebras.
\begin{lemma}  \label{W:skewform}
The bilinear form $\langle\cdot,\cdot\rangle$ on $\ga_{-1}$ is
skew-symmetric and non-degenerate.
\end{lemma}
\begin{proof}
The skew-symmetry follows by definition. The non-degeneracy
follows by the bijection $\ad e:\mathfrak{g}_{-1} \rightarrow
\mathfrak{g}_1$ and the identity $\langle x,y \rangle = (x | [y,e])$.
\end{proof}

It follows that $\ga_{-1}$ is even-dimensional. Pick a Lagrangian
(= maximal isotropic) subspace $\mathfrak l$ of
$\mathfrak{g}_{-1}$ with respect to the form
$\langle\cdot,\cdot\rangle$. Then $\dim \mathfrak l =\hf \dim
\ga_{-1}$. Introduce the following important nilpotent subalgebra
of $\mathfrak{g}$:
$$
\m := \mathfrak l \bigoplus (\bigoplus_{j \leq -2} \mathfrak{g}_j).
$$

\begin{xca}  \label{W:xca:halfdim}
Prove that
$$
\dim \m =\hf \dim \mc O_e.
$$
\end{xca}

The restriction of $\chi$ to $\m$, denoted by $\chi:\m \rightarrow
\C$, defines a one-dimensional representation $\C_\chi$ of $\m$,
thanks to the Lagraingian condition on $\lag$.
Let $I_\chi^\m$ denote the kernel of the corresponding associative
algebra homomorphism $U(\m) \rightarrow \C$, i.e. the two-sided
ideal of $U(\m)$ generated by $a -\chi(a)$ for $a \in \m$. Let
$I_{\chi}$ denote the left ideal of $U(\ga)$ generated by
$a -\chi(a)$ for $a \in \m$. Define the  induced
$\mathfrak{g}$-module (called {\em generalized Gelfand-Graev
module})
\begin{equation*}
Q_{\chi} := U(\mathfrak{g}) \otimes_{U(\m)} \C_\chi \cong
U(\mathfrak{g}) /  I_{\chi}.
\end{equation*}

The {\em finite $W$-algebra} (or simply $W$-algebra in this paper)
$\W$ is defined to be the endomorphism algebra
\begin{equation}  \label{W:def1}
\W := \End_{U(\mathfrak{g})}(Q_\chi)^{\op}.
\end{equation}
This is the original definition of Premet \cite{Pr2}.

Actually, the above notions depend in addition on $\Gamma, \lag$,
and it would be more precise to write $Q_{\chi, \Ga, \lag}$ ,
$I_{\chi, \Ga, \lag}$, $\mc W_{\chi, \Ga, \lag}$, etc. But we
choose to use the simplified notations, partly because we will
eventually show that the isoclasses of finite $W$-algebras do not
depend on $\Gamma$ and $\lag$. At different occasions later on, we
may use some other indices among ${\chi, \Ga, \lag}$ to put an
emphasis on the dependence of those indices.

\begin{example}  \label{W:ex:chi=0}
Let $e=0$. Then $\chi =0$, $\ga_0 =\ga$, $\m =0$, $Q_\chi
=U(\ga)$, and $\mc W_\chi =U(\ga)$.
\end{example}

\subsection{The Whittaker model definition}
\label{W:subsec:Whit def}

Since $Q_\chi = U(g)/I_\chi$ is a cyclic module, any endomorphism
of the $\ga$-module $Q_\chi$ is determined by the image of
$\bar{1}$, where $\bar{y}$ denotes the coset $y +I_\chi$ of $y\in
U(\ga)$. The image of $\bar{1}$ must be annihilated by $I_\chi$.
Hence we obtain
the following identification of $\W$ (as the space of Whittaker
vectors in $U(\mathfrak{g}) / I_\chi$ in the terminology
introduced in a later Section~\ref{W:subsec:functors}):
\begin{eqnarray} \label{W:def2}
\W= \{\bar{y} \in U(\mathfrak{g}) / I_\chi \mid (a -\chi(a)) y \in
I_\chi, \forall a \in \m\}.
\end{eqnarray}
By definition of $I_\chi$, $\W$ can be further identified with the
subspace of $\ad \m$-invariants in $Q_\chi$:
\begin{eqnarray} \label{W:def2b}
\W= (Q_\chi)^{\ad \m} := \{\bar{y} \in U(\mathfrak{g}) / I_\chi
\mid [a,y] \in I_\chi, \forall a \in \m\}.
\end{eqnarray}

Transferred via the above identification, the algebra structure on
$\W$ is given by
$$
\bar{y}_1 \bar{y}_2 = \overline{y_1y_2}$$ for $y_i \in
U(\mathfrak{g})$ such that $ [a,y_i] \in I_\chi$ for all $a \in
\m$ and $i=1,2$.

\begin{xca} \label{W:ex:welldefined}
Without using the identification with the definition \eqref{W:def1},
check directly that
 \begin{enumerate}
\item the ideal $I_\chi$ is $\ad \m$-invariant, hence $(Q_\chi)^{\ad \m}$ in
\eqref{W:def2b} as a vector space is well-defined;

\item for $y_i$ satisfying $ [a,y_i] \in I_\chi$ for all $a \in
\m$ and $i=1,2$, we have $ [a,y_1y_2] \in I_\chi$ for all $a \in
\m$. Hence, $(Q_\chi)^{\ad \m}$ in \eqref{W:def2b} as an algebra is well-defined.
\end{enumerate}
\end{xca}
Hence, we may regard \eqref{W:def2b}, or \eqref{W:def2}, as a
second definition of $W$-algebras.

\subsection{A simplified definition for even good gradings}
\label{W:W=m inv}

In this subsection, we will assume $\ga$ is equipped with an {\em
even} good $\Z$-grading. A good $\Z$-grading $\ga=\oplus_{j\in \Z}
\ga_j$ is called {\em even}, if $\ga_j =0$ unless $j$ is an even
integer.

All the complications of choice of a Lagrangian/isotropic subspace
$\lag$ disappear for even $\Z$-gradings, since $\ga_{-1} =0$. Now
we have $\mathfrak{m} = \oplus_{j \leq -2} \mathfrak{g}_j$, and
a parabolic submodule
\begin{equation} \label{W:def:p}
 \mathfrak{p} := \bigoplus_{j \geq 0} \mathfrak{g}_j.
\end{equation}
It follows by the PBW theorem that
\begin{equation*}
U(\mathfrak{g}) = U(\mathfrak{p}) \bigoplus I_\chi.
\end{equation*}
The projection $\pr_\chi:U(\mathfrak{g}) \rightarrow
U(\mathfrak{p})$ along this direct sum decomposition induces an
isomorphism
$$
\overline{\pr}_\chi: U(\mathfrak{g}) / I_\chi
\stackrel{\sim}{\longrightarrow} U(\mathfrak{p}).
$$

By the above observation,  the algebra $\W$
can be regarded as a {\em subalgebra} of $U(\mathfrak{p})$. Consider
a $\chi$-twisted adjoint action of $\mathfrak{m}$ on
$U(\mathfrak{p})$ by $a . y := \pr_\chi ([a,y])$ for $a \in
\mathfrak{m}$ and $y \in U(\mathfrak{p})$. We identify $\W$ as
\begin{eqnarray}  \label{W:def:subalg}
\W = U(\mathfrak{p})^{\ad \m} := \{y \in U(\mathfrak{p}) \mid
[a,y] \in I_\chi, \forall a \in \m \}.
\end{eqnarray}
Since $\overline{\pr}_\chi: U(\mathfrak{g}) / I_\chi
\stackrel{\sim}{\rightarrow} U(\mathfrak{p})$ is an isomorphism of
$\m$-modules, the equivalence of \eqref{W:def:subalg} with the second
definition (for even $\Z$-gradings) is clear.

Hence, we may take \eqref{W:def:subalg} as a third definition of
the $W$-algebras. This is the original definition used by Kostant
and Lynch \cite{Ko, Ly} (who only considered even Dynkin
$\Z$-gradings). As observed by Brundan-Kleshchev \cite{BK1}, it
works equally well in the generality of even good gradings by
dropping ``Dynkin". Brundan-Goodwin-Kleshchev \cite[Section~2]{BGK} also has
a (necessarily) more complicated version of this definition
without assuming the grading to be even.

\begin{example}  \label{W:ex:regNil}
According to a theorem of Kostant (see Theorem~\ref{W:Kostant}
below), for a regular nilpotent element $e$ in $\ga$, $\mc W_\chi$
is isomorphic to $\mc Z (\ga)$, the center of $U(\ga)$.

Let $e=E_{12} \in \gl_2$, which is regular nilpotent. Then $\m =\C
f$ with $f =E_{21}$, and $\mf p =\C e +\C E_{11} +\C E_{22}$. A
direct computation shows that $E_{11} +E_{22}$ and $e +\frac14 h^2
-\hf h$ lie in $U(\mf p)^{\ad \m}$. (If we combine with Kostant's
theorem, then $\W =U(\mf p)^{\ad \m}$ is the polynomial algebra
generated by these two elements.)

\medskip

Let $\ga =\gl_n$ and $e =J_n$ as in Example~\ref{W:ex:regular}. Let
\begin{equation*}
\Omega(u) := \left(
\begin{array}{cccccc}
E_{11}+u+ 1 & E_{12} &  &\cdots & E_{1n}\\
1&E_{22}+u+ 2&E_{23}&&\vdots\\
0& &\ddots&\ddots & \vdots \\
\vdots&& \ddots & \ddots 
&E_{n-1,n}\\
0&\cdots&0&1&E_{nn}  +u+ n
\end{array}
\right)
\end{equation*}
The row determinant of a matrix $A = (a_{i,j})_{n \times n}$
with non-commutative entries is defined to be
\begin{equation}\label{W:detdefr}
\text{rdet} \, A  = \sum_{\pi \in S_n} \text{sgn}(\pi) a_{1,\pi 1}
\cdots a_{n,\pi n}.
\end{equation}
Write $\text{rdet} \, \Omega(u) =u^n +\sum_{i=1}^{n} w_i u^{n-i}$
for $w_i \in U(\mf p)$. As shown in \cite[Section~12]{BK1}, $w_i,
1\le i\le n$ are commuting elements which lie in $U(\mf p)^{\ad
\m}$, and they freely generate $\W = U(\mf p)^{\ad \m}$.
\end{example}

\subsection{The BRST definition}
\label{W:subsec:BRST}

There is yet another definition of $W$-algebras using the BRST
complex.

For the sake of simplicity, we will work under the assumption of
an {\em even} good $\Z$-grading $\Gamma: \ga = \oplus_{j \in \Z}
\ga_j$ for $e \in \ga_2$. In this case, we have $\m = \oplus_{j
\leq -2} \ga_j$.

Let us take another copy of $\m$, which will be denoted by $\hat
\m$.

Endow the vector space $\m^* \oplus \hat{\m}$ with the symmetric
bilinear form induced by the pairing between $\m^*$ and $\hat\m$.
The corresponding Clifford algebra on $\m^* \oplus \hat{\m}$ can
be naturally identified with $\wedge (\m^*) \otimes \wedge
(\hat{\m})$. Consider the tensor algebra of this Clifford algebra
with $U(\ga)$
$$
\brst  : = \wedge (\m^*) \otimes U(\ga) \otimes \wedge (\hat{\m}).
$$
It admits a BRST (cohomological) $\Z$-grading with the assignment
of degrees to generators in $\m^*, \ga, \hat{\m}$ to be $1, 0,
-1$, respectively. This $\Z$-grading is compatible with the
superalgebra structure on $\brst$ by declaring generators in $\ga$
to be even and  generators in $\m^* \oplus \hat{\m}$ odd.

Take a basis $\{b_i\}$ of $\m$
and let $\{f^i\}$ be its dual basis for $\m^*$.
%
%
Let $d =[\phi, -]$ be the super-derivation of $\brst$ (of BRST
degree $1$), which is defined to be the supercommutator with the
following odd element
$$
\phi =f^i (b_i -\chi (b_i)) -\hf f^i f^j [b_i, b_j]^\wedge
$$
where $a^\wedge$ 
denotes the corresponding element in $\hat{\m}$ for $a\in \m$.
Here and below we have adopted the convention of summation over
repeated indices. It is easy to check that $\phi$ is independent
of the choices of the dual bases.

In more concrete terms, one finds that, for $x \in \ga, f \in
\m^*$ and $a^\wedge \in \hat \m$ corresponding to $a \in \m
\subseteq \ga$,
\begin{align} \label{W:d formula}
\begin{split}
d(x) =  f^i [b_i, x], & \qquad
d(f)  = {\textstyle\frac{1}{2}}f^i \, \text{ad}^* b_i (f), \\
d(a^\wedge) =& a - \chi(a)   + f^i [b_i, a]^\wedge,
\end{split}
\end{align}
where we have denoted by $\text{ad}^*$ the coadjoint action.

\begin{xca}  \label{W:xca:diff}
Verify that $d^2 =0$.
\end{xca}

So we have defined the so-called BRST complex $(\brst, d)$ and we
can define as usual its cohomology $H^* (\brst)$. The BRST
definition of the $W$-algebra (cf. de Boer and Tjin \cite{BT}) is
$$
\W =H^0(\brst, d).
$$
This is justified by the following isomorphism theorem, which
confirms a conjecture of Premet \cite[1.10]{Pr2}.

\begin{theorem} \cite{DHK} \label{W:th:BRST=Wh}
The BRST cohomology $H^* (\brst)$ is concentrated in cohomological
degree $0$ and there is an algebra isomorphism between $H^0(\brst,
d)$ and $(Q_\chi)^{\ad \m}$.
\end{theorem}

\begin{proof}[Sketch of a proof]
Note that in the definition of the BRST complex $(\brst, d)$,
$\brst$ does not depend on $\chi$, but $d$ does. We introduce a
complex $(\brst_\chi, d')$ which is isomorphic to the BRST complex
$(\brst, d)$ but with a shift of the dependence of $\chi$ from the
differential to the complex.

The complex $\brst_\chi :=\wedge (\m^*) \otimes (U(\ga) \otimes
\C_\chi) \otimes \wedge (\hat{\m})$ is obtained from twisting
$\brst$ by $\chi$, and $U(\ga) \otimes \C_\chi$ is naturally an
$\m$-bimodule; the differential $d'$ equals the sum of two
(anti)commuting differentials $d_\m$ and $d^\m$, where $d_\m$ (and
respectively, $d^\m$) is a $\m$-homology (respectively,
$\m$-cohomology) differential. Then $(\brst_\chi, d)$ can be
regarded as a double complex, and the first term of a spectral
sequence computation gives us $H^{\bullet,0} (\brst_\chi, d_\m)
\cong \wedge^\bullet (\m^*) \otimes Q_\chi$ (the $\m$-cohomology
complex with coefficient in $Q_\chi$); moreover, $H^{p,q}
(\brst_\chi, d^\m) =0$ whenever $q \neq 0$. The spectral sequence
stabilizes at the second term, which shows $H^i (\brst_\chi, d)
\cong  \gr H^i(\brst_\chi, d) = H^i (\m, Q_\chi)$, for all $i$.

We shall see from Theorem~\ref{W:th:2steps} later on that $H^0 (\m,
Q_\chi) =Q_\chi^{\ad \m}$ and that $H^i (\m, Q_\chi) =0$ for
$i>0$.

A quasi-isomorphism $\wedge^\bullet (\m^*) \otimes Q_\chi
\rightarrow \brst_\chi$ can then be constructed explicitly, and it
induces the algebra isomorphism on $0$th cohomology. Putting all
things together, we have obtained the theorem.
\end{proof}

We shall describe such an isomorphism in a more concrete fashion.
Note by the PBW theorem that
$
\mc B^0 = U( {\ga}) \oplus \m^* \mc B^0 \hat{\m},
$
with $\m^* \mc{B}^0 \hat{\m}$ being a two-sided ideal. So we can
consider the composition $q$ of the two natural maps $ \mc B^0
\twoheadrightarrow U( {\ga}) \twoheadrightarrow Q_\chi, $ which
fits into the following commutative diagram:

\begin{equation*}
\begin{CD}
0 & @>>> & I_\chi \oplus\m^* \mc B^0 \hat \m & @>>> & \mc{B}^0
&@>q>> & Q_\chi & @>>> 0\\
&&&&  @VVV &&@VV\text{pr}V&&  @|\\
0 & @>>> & I_\chi   & @>>> & U (\ga) &@>>> & Q_\chi & @>>> 0
\end{CD}
\end{equation*}
By \eqref{W:d formula}, we check that $d$ maps $\mc B^{-1}$ into
$I_\chi \oplus \m^* \mc B^0 \hat \m$. So   $q$
induces a well-defined linear map
$
q:H^0(\brst, d) \rightarrow Q_\chi.
$
Then, $q$ is the algebra isomorphism between $H^0(\brst, d)$ and
$Q_\chi^{\m}$ claimed in Theorem~\ref{W:th:BRST=Wh}.

\begin{remark}  \label{W:BRSTgeneral}
The BRST formulation of $W$-algebras works for an arbitrary good
$\Z$-gradings without the assumption of ``even". One needs to add
an additional ``neutral fermion" tensor factor to $\brst$ which
corresponds to the presence of nonzero $\ga_{-1}$ (cf. \cite{KRW,
DHK}). Again here, one has the flexibility of choosing a
Lagrangian/isotropic subspace $\lag$\ of $\ga_{-1}$ in defining
the modified BRST complex.
\end{remark}

\begin{remark}  \label{W:defLosev}
There is yet another equivalent definition of finite $W$-algebras, due to
Losev \cite{Lo1}, via deformation quantization.
\end{remark}

\section{Quantization of the Slodowy slices}
\label{W:sec:Slodowy}

In this section, we will explain the independence of $\lag$ in the
definition of finite $W$-algebras, denoted now by $\mc W_\lag$,
associated to an isotropic subspace $\lag \subseteq \ga_{-1}$,
following Gan and Ginzburg \cite{GG}. To that end, we explain the
geometric picture behind the $W$-algebras in terms of Slodowy
slices, due to Premet \cite{Pr2} (and generalized in \cite{GG}).
Everything works in a general good $\Z$-grading setting, as
remarked in the Introduction of \cite{BK1}.
\subsection{The isotropic subspace definition of $W$-algebras}
\label{W:subsec:GGdef}

Fix an isotropic subspace $\lag$ of $\ga_{-1}$ with respect to
$\langle\cdot,\cdot\rangle$, which means that $\langle \lag, \lag
\rangle =0$, and let
$$
\lag' =\{x \in \ga_{-1} \mid \langle x, \lag \rangle =0 \}.
$$
Clearly we have $\lag \subseteq \lag'$. In this section, we will
introduce and study the generalized Gelfand-Graev modules and
$W$-algebras (using notations $Q_\lag$ and $\mc W_\lag$ to
emphasis the dependence on $\lag$) in such a setting.

Define the following nilpotent subalgebras of $\ga$ (where $\m
\subseteq \m'$):
\begin{displaymath}
\m = \lag \bigoplus (\bigoplus_{i \leq -2} \ga_i) \qquad
\textrm{and}
 \qquad
\n = \lag' \bigoplus (\bigoplus_{i \leq -2} \ga_i).
\end{displaymath}
Note that $\chi$ restricts to a character on $\m$ (actually this
is equivalent to the requirement that $\lag$ is isotropic in
$\ga_{-1}$). Denote by $\C_{\chi}$ the corresponding 1-dimensional
$U(\m)$-module, and let
$$
Q_{\lag}=U(\ga) \bigotimes_{U(\m)}
\C_{\chi} =U(\ga)/I_\lag
$$
be the induced $U(\ga)$-module, where $I_\lag$ denotes the left
ideal of $U(\ga)$ generated by $a-\chi(a),\forall a \in \m$. The
same proof for Exercise~\ref{W:ex:welldefined} shows that
$I_{\lag}$ is $\ad\n$-invariant. Thus, there is an induced $\ad
\n$-action on $Q_{\lag}$.

Let
$$
\mc W_{\lag} := (U(\ga)/I_\lag)^{\ad \n} \equiv \{\bar{y} \in
U(\ga)/I_\lag \mid [a, y] \in I_{\lag}, \forall a \in \n\}.
$$
In the same spirit of  Exercise~\ref{W:ex:welldefined}, letting
$\bar{y}_1 \bar{y}_2 = \overline{y_1y_2},$ for $\bar{y}_1,
\bar{y}_2 \in \mc W_{\lag}$, provides a well-defined algebra
structure on $\mc W_{\lag}$.

Clearly, the above definition of $Q_\lag$ and $\mc W_\lag$ reduces
to the earlier one when $\lag$ is Lagrangian.

\begin{theorem} \cite{GG}  \label{W:th:W lag}
The algebras $\mc W_\lag$ are all isomorphic for different choices
of isotropic (in particular, Lagrangian) subspaces $\lag \subseteq
\ga_{-1}.$
\end{theorem}
The proof of this theorem requires some preparations.

\subsection{$\Gamma$-graded $\sgl_2$-triples}
\label{W:subsec:sl2graded}

\begin{lemma} \label{W:lem:EK}
Let $\Ga: \ga =\oplus_{j\in \Z} \ga_j$ be a  good $\Z$-grading for
$0 \neq e \in \ga_2$.  Then there exists $h \in \ga_0$ and $f \in
\ga_{-2}$ such that $\{ e,h,f \}$ form an $\sgl_2$-triple.
\end{lemma}

This will be referred to as  a {\em $\Gamma$-graded
$\sgl_2$-triple} associated to $e$. This $\Gamma$-graded
$\sgl_2$-triple will be used often.

\begin{proof}
By Jacobson--Morozov Theorem in \ref{W:JacobsonM}, there exists an
$\sgl_2$-triple $\{ e,h',f'\}$ in $\ga$. Denote  by $h'=\sum_{j
\in \Z} h_j$, $f'=\sum_{j \in \Z}f'_j$ the decompositions with
respect to the $\Z$-grading $\Gamma$, and take $h=h_0 \in \ga_0$.
Then $[h,e]=2e$ and $ h =[e,f'_{-2}] \in [e,\ga]$.
Taking
$\tilde{f}$ to be the degree $-2$ component of $f'_{-2}$ with respect to
$\ad h$, we obtain a new $\sgl_2$-triple $\{ e,h,\tilde {f}\}$.
Finally, taking $f$ to be the degree $-2$ component of $\tilde{f}$
with respect to $\Gamma$, we obtain the desired $\sgl_2$-triple
$\{ e,h,  {f}\}$. (Indeed, $\tilde{f} =f$ by Lemma~\ref{W:rigid};
in other word, $\tilde f$ is already $\Ga$-homogeneous.)
\end{proof}

%
%

Given a subspace $V$ of $\ga$, we let
$$
V^\perp =\{x\in \ga \mid
(x| v) =0, \forall v\in V\},
\quad
V^{*,\perp} =\{\xi \in
\ga^* \mid \xi (v) =0, \forall v\in V\}.
$$

\begin{lemma}\label{W:tangent}
Let  $\{ e,h,f \}$ be a $\Gamma$-graded $\sgl_2$-triple. Then, we
have $\m^\perp = [\n,e] \oplus \ga_f$.
\end{lemma}

\begin{proof}
We have properties \eqref{W:deg=2}-\eqref{W:numbersimple} associated
to $e$ and their $f$-counterparts available. Now the lemma follows
from the four facts below:
\begin{itemize}
\item [(i)] $\m^\perp \supseteq \ga_f$. This follows from
$\m^\perp \supseteq \ga_{\leq 0}$ by \eqref{W:pairing} and  the
$f$-counterpart to \eqref{W:vanish} which says that $\ga_f
\subseteq \ga_{\leq 0}$. \item [(ii)] $\m^\perp \supseteq [\n,e]$.
This can be seen by a direct computation: $(m |[m',e]) =([m,m']
|e) =0$ for $m \in \m$ and $m' \in \m'$. \item [(iii)] $[\n,
e]\cap \ga_f =0$. This follows by the $\sgl_{2}$ representation
theory. \item [(iv)] $\dim\a^{\perp}= \dim\n +\dim \ga_0
+\dim\ga_{-1} =\dim [\n, e]+\dim \ga_f $. This follows by the
bijection $\,\n\to [\n, e]\,,\, x\mapsto [x,e],$ by \eqref{W:inj},
and  the $f$-counterpart to \eqref{W:numbersimple}.
\end{itemize}
\end{proof}

\subsection{Some $\C^*$-actions}
\label{W:C*}

The embedding of the $\Gamma$-graded $\sgl_2$-triple in $\ga$
exponentiates to a rational homomorphism $\tilde{\gamma}:
\text{SL}_{2}(\C) \rightarrow G$. We put
\begin{displaymath}
\gamma : \C^{*} \rightarrow G,  \quad  t \mapsto \tilde{\gamma}
(\text{diag}\, (t,t^{-1})).
\end{displaymath}
Note that $\bigl(\Ad \gamma(t)\bigr)(e) = t^{2} e$. The desired
action of $\C^{*}$ on $\ga$, to be denoted by $\rho$, is defined
by
$$
\rho(t)(x)=t^{2}\cdot \bigl(\Ad \gamma(t^{-1})\bigr)(x), \quad
\forall t \in \C^{*}, x \in \ga.
$$
Note that $\rho(t)(e+x) = e + \rho(t)(x)$. Thus, since $\rho(t)$
stabilizes $\ga_f$ and $\m^{\perp}$, it also stabilizes $e+ \ga_f$
and $e +\m^{\perp}$, respectively. Note that $\underset{^{t\to 0}}{\lim}^{\,}
\rho(t)(x) = e, \forall x \in e+ \m^{\perp}$, i.e. the
$\C^{*}$-action on $e+ \m^{\perp}$ is contracting.

Let $\texttt M'$ denote the closed subgroup of $G$ whose Lie
algebra is $\m'$. Define a $\C^{*}$-action on $\texttt M' \times
(e+ \ga_f)$ by
\begin{displaymath}
t \cdot (g, x) = (\gamma(t^{-1})g\gamma(t), \rho(t)(x)).
\end{displaymath}
Note that for any $(g, x) \in \texttt M' \times (e+ \ga_f)$, we
have $\underset{^{t\to 0}}{\lim}^{\,} t \cdot (g,x) = (1,e)$.

\subsection{}
\label{W:slice}

Denote by $\kappa: \ga \rightarrow \ga^*$ the isomorphism induced
by the non-degenerate bilinear form $(\cdot|\cdot)$. Following the
terminology of Gan and Ginzburg, we will call
$$
 \mc S := \chi +\ker \ad^* f \equiv \kappa(e+ \ga_f)
$$
the {\em Slodowy slice} (through $\chi$), or $e+ \ga_f$ the {\em
Slodowy slice} (through $e$). Note that $e+ \ga_f$ is a
transversal slice to the adjoint orbit through $e$.

\begin{lemma} \label{W:isom}
The adjoint action map
\begin{displaymath}
\alpha : \texttt M' \times (e+ \ga_f) \longrightarrow e+
\a^{\perp}
\end{displaymath}
is an isomorphism of affine varieties.
\end{lemma}

\begin{proof}[Sketch of a proof]
The action map $\alpha$ is  $\C^{*}$-equivariant.

By Lemma~\ref{W:tangent}, the differential map of $\alpha$ is an
isomorphism between the tangent spaces at the $\C^*$-fixed points
$(1,e)$ and $e$.

Now, Lemma \ref{W:isom} follows from the following general
nonsense: {an equivariant morphism $\alpha:X_{1} \rightarrow
X_{2}$ of smooth affine $\C^{*}$-varieties  with contracting
$\C^{*}$-actions which induces an isomorphism between the tangent
spaces of the $\C^{*}$-fixed points must be an isomorphism.}
\end{proof}

\subsection{}
\label{W:subsec:diag}

The remainder of this section is to make sense the following
commutative diagram:
\begin{equation}\label{W:QSlodowy}
\begin{CD}
\gr U(\mathfrak{g}) & @= &\C[\mathfrak{g}^*] &@>\sim>> & \C[\mathfrak g]\\
@VVV &&@VVV&& @VVV\\
\gr Q_\lag &@= &\C[\chi+\mathfrak{m}^{*,\perp}] &@>\sim>> &
\C[e+\mathfrak{m}^\perp]
\\
@AAA&&@VVV&&@VVV\\
\gr \mc W_\lag &@>\sim>\nu>&\C[\mc S]&@>\sim>>&\C[e+\ga_f]
\end{CD}
\end{equation}

\begin{remark}  \label{W:rem:reminder}
When reading \eqref{W:QSlodowy}, it is instructive to keep in mind
the isomorphism $\C[e+\ga_f] \cong
\C[e+\mathfrak{m}^\perp]^{\text{Ad}\,\texttt M'}$ by
Lemma~\ref{W:isom} and the identity $\mc W_\lag =Q_\lag^{\ad \m'}$ by
definition. Another way to rephrase this isomorphism
is that $\mc S$ is obtained
by the symplectic reduction for the coadjoint action of $\texttt M'$
on $\ga^*$ at the point $\chi$.
\end{remark}

Some parts of the diagram \eqref{W:QSlodowy} are easy to explain.
The vertical arrows in the third column are restriction maps
(e.g., $\ga_f \subseteq \m^\perp$ by Lemma~\ref{W:tangent}). The
second column (which is more conceptual from the coadjoint orbit
philosophy) is transferred from the (simpler) third one by the
isomorphism $\kappa: \ga \rightarrow \ga^*$.

It remains to explain the first column (where $\gr A$ stands for
the associated graded space/algebra for a Kazhdan-filtered
space/algebra $A$) and its identification with the second one. See
Section~\ref{W:Kgrading} below.

Granting the diagram \eqref{W:QSlodowy}, we can complete the proof
of Theorem~\ref{W:th:W lag}.
\begin{proof}[Proof of Theorem~\ref{W:th:W lag}]
Given two isotropic subspaces $\lag_1, \lag$ of $\ga_{-1}$ such
that $\lag_1 \subseteq \lag$, we have a natural
$U(\ga)$-homomorphism $Q_{\lag_1} \rightarrow Q_{\lag}$, which
gives rise to an algebra homomorphism $\mc W_{\lag_1} \rightarrow
\mc W_{\lag}$. The associated graded map $\gr \mc W_{\lag_1}
\rightarrow \gr \mc W_{\lag}$ is an algebra isomorphism, by
Theorem~\ref{W:th:qunantiz}. Hence the canonical map $\mc
W_{\lag_1} \rightarrow \mc W_{\lag}$ is indeed an isomorphism.

Taking $\lag_1 =0$ implies that the isoclass of the algebra $\mc W_{\lag}$ is
independent of the choice of an isotropic subspace $\lag$. This
completes the proof of the theorem.
\end{proof}
\subsection{Kazhdan grading and filtration}
\label{W:Kgrading}

Let $\{U_j(\ga)\}_{j\leq 0}$ be the standard PBW filtration on
$U(\ga)$. Recall that $h_\Gamma$ is introduced in
Lemma~\ref{W:grading}. The action of $\ad h_\Gamma$ induces a
grading on each $U_j(\ga)$ by
$$ U_j(\ga)_i = \{x \in U_j (\ga) \mid \ad h_\Gamma (x) =ix \}.$$
The Kazhdan filtration on $U(\ga)$ is defined by letting
$$
F_n U(\ga) =\sum_{i+2k \leq n} U_k(\ga)_i,
$$
and it enjoys the following favorable properties:

\begin{enumerate}
\item[(a)] The canonical map $\gr U(\ga) \rightarrow S[\ga]
=\C[\ga^*]$ is an isomorphism of graded commutative algebras.
Indeed, for $x \in \ga_i, y \in \ga_j$, we have $x \in F_{i+2}
U(\ga), y \in F_{j+2} U(\ga)$, and $[x,y] \in F_{i+j+2} U(\ga)$.

The Kazhdan grading on $\ga$ and so on $S(\ga)$ (which is
compatible with the one on $\gr U(\ga)$) can be described
directly.

\item[(b)] There is a Kazhdan filtration $\{F_n Q_\lag\}$ on
$Q_\lag =U(\ga)/I_\lag$ induced from $U(\ga)$.

\begin{itemize}
\item $F_n Q_\lag =0$ unless $n \geq 0$. Indeed, the generators
$\{a -\chi(a) \mid \forall a \in \m\}$ of the ideal $I_\lag$
contains all the negative-degree generators of $U(\ga)$.

\item $\gr Q_\lag =\gr U(\ga) /\gr I_\lag$ is a commutative
$\N$-graded algebra.

\item The ideal $\gr I_\lag$ in $\gr U(\ga) = \C[\ga^*]$ can be
identified with the ideal of polynomial functions on $\ga^*$ which
vanish on $\chi +\m^{*,\perp}$.

\item The canonical map $\gr Q_\lag \rightarrow
\C[\chi+\mathfrak{m}^{*,\perp}]$ is an algebra isomorphism.
\end{itemize}

\item[(c)] There is an induced Kazhdan filtration on the subspace
$\mc W_\lag$ of $Q_\lag$ such that $F_n \mc W_\lag =0$ unless $n
\geq 0$.
\end{enumerate}

\begin{remark}  \label{W:rem:poisson}
One could define the graded algebra structures on the second and
third columns directly by some canonical $\C^*$-action on $\ga$
which preserves $e +\m^\perp$ and $e+\ga_f$ (similar to
Section~\ref{W:C*}), and then show that the horizontal maps
between the last two columns in \eqref{W:QSlodowy} are
isomorphisms of graded algebras.

We have chosen not to discuss the  canonical Poisson algebra
structures for various graded algebras above (see \cite{GG} for
details).
\end{remark}

\subsection{}
\label{W:graded=}

To complete the commutative diagram \eqref{W:QSlodowy}, we define
the map $\nu: \gr \mc W_\lag \rightarrow \C[\mc S]$ to be the
composite of the three natural maps
$$
\gr \mc W_{\lag} \longrightarrow \gr Q_\lag \longrightarrow
\C[\chi +\m^{*,\perp}] \longrightarrow \C[\mc S].
$$
The last piece for the completion of the diagram
\eqref{W:QSlodowy} is the following result (due to Premet
\cite{Pr2} for Lagrangian $\lag$ and Gan-Ginzburg \cite{GG} for
isotropic $\lag$), whose proof will be postponed to the next
section.

\begin{theorem}  \label{W:th:qunantiz}
The map $\nu: \gr \mc W_\lag \rightarrow \C[\mc S]$ is an
isomorphism of [graded] algebras.
\end{theorem}

\subsection{} \label{W:subsec:Kostant}

We have the following fundamental theorem of Kostant \cite{Ko}.
Here we follow the approach of Premet \cite{Pr2}.

\begin{theorem} \label{W:Kostant}
Let $e$ be a regular nilpotent element in $\ga$. Then $\mc W_\lag
\cong \mc Z (\ga)$, the center of $U(\ga)$.
\end{theorem}

\begin{proof}[Sketch of a proof]
The algebra of invariants $S(\ga^*)^G$ is known to be a polynomial
algebra  in $r =\text{rank} (\ga)$ variables. Let $f_1, \ldots,
f_r$ denote the algebraically independent homogenous generators of
$S(\ga^*)^G$, with $\deg f_i =m_i +1$, where $m_i$ denotes the
exponents of $\ga$. Since the associated graded algebra of $\mc Z
(\ga)$ with respect to the PBW filtration $\{U_j(\ga)\}_{j \ge 0}$
is isomorphic to $S(\ga^*)^G$ (where we have identified $\ga \cong
\ga^*$), we can choose lifts $\tilde f_i \in U_{m_i+1}(\ga)\cap
\mc Z (\ga)$ of each $f_i$.

Restricting the adjoint quotient map $\ga \rightarrow \ga//G
=\C^r$, $x \mapsto (f_1(x), \ldots, f_r(x))$ to the Slodowy slice
gives rise to a $\C^*$-equivariant isomorphism of affine varieties
from $e+\ga_f$ to $\C^r$ \cite{Ko, Slo}, where the $\C^*$-action
$\tilde{\rho}$ on $e+\ga_f$ is such that $\tilde{\rho}(t^2) =\rho
(t)$ for $t\in \C^*$. Hence, $\gr \mc W_\lag$ is generated by (the
restrictions of) ${f_i}$ for $1\le i \le r$.

One can show that the restriction of $U(\ga) \rightarrow
\End_\C(Q_\lag)$ to the center $\mc Z(\ga)$ is always injective,
hence we obtain a monomorphism $\mc Z (\ga) \rightarrow \mc W_\lag
=\End_\ga(Q_\lag)$ by \eqref{W:def1}, whose associated graded map
is an isomorphism. If follows from a filtered algebra argument
that this has to be an isomorphism.
\end{proof}

\begin{remark} \label{W:subreg}
The $W$-algebras associated to subregular nilpotent elements are
particularly interesting, as it can be viewed as a noncommutative
deformation of the simple singularities. See Premet \cite{Pr2} and
Gordon-Rumynin \cite{GR} for details.
\end{remark}
\section{An equivalence of categories}
\label{W:sec:cat=}

In this section, we formulate an equivalence of categories due to
Skryabin. The proof of this category equivalence here, due to
\cite{GG}, uses similar ideas of the proof of
Theorem~\ref{W:th:qunantiz}, which we will first complete.
\subsection{Proof of Theorem~\ref{W:th:qunantiz}}
\label{W:subsec:proof}

Recall that $\n$ is graded with respect to the grading $\Gamma$.
We view $U(\ga)$ and $Q_\lag$ as a $\n$-modules via the adjoint
$\n$-action. Then, $U(\ga)$ and $Q_\lag$ are Kazhdan-filtered
$\n$-modules and the canonical map $p:U(\ga) \rightarrow Q_\lag$
is $\n$-module homomorphism. Thus, $\gr U(\ga)$ and $\gr Q_\lag$
are Kazhdan-graded $\n$-modules, and $\gr \ p: \gr U(\ga)
\rightarrow \gr Q_\lag$ is an $\n$-module homomorphism.

Note that $\mc W_{\lag} = H^{0}(\n, Q_\lag)$, the $0$th Lie
algebra cohomology of $\n$ with coefficient in $Q_\lag$. We
reformulate and prove Theorem~\ref{W:th:qunantiz} as follows.
\begin{theorem} \cite{GG}   \label{W:th:2steps}
The map $\nu:  \gr \mc W_\lag \rightarrow \C[\mc S]$ is equal to the
composite $\nu_2 \nu_1$:
$$
\gr H^{0}(\n, Q_\lag) \stackrel{\nu_1}{\longrightarrow}
H^{0}(\n,\gr Q_\lag) \stackrel{\nu_2}{\longrightarrow} \C[\mc S]
$$
where $\nu_1$ and $\nu_2$ are isomorphisms. Moreover, $H^{i}(\n,
Q_\lag) = H^{i}(\n,\gr Q_\lag) = 0$ for $i>0$.
\end{theorem}

\begin{proof}
By (b) of \ref{W:Kgrading} and Lemma~\ref{W:isom}, we obtain
isomorphisms of vector spaces
$$
\gr Q_\lag \cong \C[\chi+\mathfrak{m}^{*,\perp}] \cong
\C[\texttt M'] \otimes \C[\mc S].
$$
These isomorphisms are actually on the level of $\n$-modules,
where the $\n$-module structure on the third space comes from the
$\n$--adjoint action on its first tensor factor $\C[\texttt M']$.
Now
$$
H^i (\n,\C[\texttt M']) =\delta_{i,0}\C
$$
since the standard cochain complex for Lie algebra cohomology with
coefficients in $\C[\texttt M']$ is just the algebraic de Rham
complex for $\texttt M'$ which admits trivial cohomology for an
affine space such as $\texttt M'$. Putting all things together, we
have proved that $\nu_2$ is an isomorphism and that $H^{i}(\n,\gr
Q_\lag) = 0$ for $i>0$.

\smallskip
Recall from (b) in \ref{W:Kgrading} that the Kazhdan-filtration on
$Q_\lag$ has no negative-degree component. Note in addition that
$\n$ is a negatively graded subalgebra of $\ga$ with respect to
the grading $\Gamma$, and so its dual $\n^{*}$ is positively
graded (with respect to $\Gamma$). We write this graded
decomposition as $\n^{*} = \bigoplus_{i \geq 1}\n^{*}_i$.

Consider the standard cochain complex for computing the
$\n$-cohomology of $Q_\lag$:
\begin{equation} \label{W:eqn:Q}
0 \longrightarrow Q_\lag \longrightarrow \n^{*} \otimes Q_\lag
\longrightarrow \ldots \longrightarrow \wedge^{k}\n^{*} \otimes
Q_\lag \longrightarrow \ldots.
\end{equation}
A filtration on $\wedge^{k}\n^{*} \otimes Q_\lag$ is defined by
letting $F_{p}(\wedge^{k}\n^{*} \otimes Q_\lag)$ be the subspace
of $\wedge^{k}\n^{*} \otimes Q_\lag$ spanned by $(x_{1} \wedge
\ldots \wedge x_{k}) \otimes v$, for all $x_{1} \in
\n^{*}_{i_{1}}, \ldots, x_{k} \in \n^{*}_{i_{k}}$ and $v \in
F_{j}Q_\lag$ such that $i_{1} + \ldots + i_{n} + j \leq p$. This
defines a filtered complex structure on (\ref{W:eqn:Q}).

Taking the associated graded complex of (\ref{W:eqn:Q}) gives the
standard cochain complex for computing the $\n$-cohomology of $\gr
Q_\lag$.

Now consider the spectral sequence with
\begin{displaymath}
E_{0}^{p,q} = F_{p}(\wedge^{p+q}\n^{*} \otimes Q_\lag)/
F_{p-1}(\wedge^{p+q}\n^{*} \otimes Q_\lag).
\end{displaymath}
Then $E_{1}^{p,q} = H^{p+q}(\n, \gr_{p} Q_\lag)$, and the spectral
sequence converges to $E_{\infty}^{p,q} = F_{p}H^{p+q}(\n,
Q_\lag)/ F_{p-1}H^{p+q}(\n, Q_\lag)$. The remaining parts of
Theorem~\ref{W:th:2steps} follow from this and the parts about
$\gr Q_\lag$ established above.
\end{proof}

\subsection{The Whittaker functor}
\label{W:subsec:functors}

As we have established the independence of the $W$-algebras from
the choice of  isotropic subspaces $\lag$, we will switch the
notations for the generalized Gelfand-Graev module and the
$W$-algebra back to $Q_\chi, \mc W_\chi$, to emphasize the crucial
dependence on $\chi$.

In the remainder of this section, we will set up the connections
between $W$-algebras and the category of Whittaker modules. To
that end, we shall fix an nilpotent element $e$ (and hence $\chi$)
and a {\em Lagrangian} subspace $\lag$ of $\ga_{-1}$ once for all.

\begin{definition}  \label{W:def:Whit}
A $\ga$-module $E$ is called a {\em Whittaker module} if $a
-\chi(a)$, $\forall a \in \m$, acts on $E$ locally nilpotently. A
{\em Whittaker vector} in a Whittaker $\ga$-module $E$ is a vector
$x \in E$ which satisfies $(a -\chi(a)) x=0, \forall a \in \m. $

Let $\whmod$ be the category of finitely generated Whittaker
$\ga$-modules.
\end{definition}
Denote the subspace of all Whittaker vectors in $E$ by
$$
\wh (E) = \{v\in E \mid (a -\chi(a)) v=0, \forall a \in \m\}.
$$
Recall the second definition $\W =(U(\ga)/I_\chi)^{\ad \m}$, and
we denote by $\bar{y} \in U(\ga)/I_\chi$ the coset associated to
$y \in U(\ga)$.
\begin{lemma}  \label{W:welldefined}
\begin{enumerate}
\item Given a Whittaker $\ga$-module $E$ with an action map
$\varrho$, $ \wh (E)$ is naturally a $\W$-module by letting
$\bar{y}. v =\varrho(y)v$ for $v\in \wh (E)$ and $\bar{y} \in \W$.

\item For $V \in \Wmod$, $Q_\chi \otimes_{\W}V$ is a Whittaker
$\ga$-module by letting
$$\quad y. (q \otimes v) =(y.q) \otimes v, \quad \text{ for }
y \in U(\ga), q \in Q_\chi =U(\ga)/I_\chi, v \in V.
$$
\end{enumerate}
\end{lemma}

\begin{proof}
(1) Let $v \in \wh (E)$. The formula $\bar{y}. v =\varrho(y)v$ is
well-defined since $\varrho(y) v=0$ for all $y \in I_\chi$ and $v
\in \wh (E)$. Being $\ad \m$-invariants, $\bar{y} \in \W$
satisfies $[\varrho (y), a] \in I_\chi$ and so $\bar{y}. v \in \wh
(E)$. It follows from the $\ga$-module homomorphism $\varrho$ that
the formula defines an action of $\W$.

(2) When we use the first definition $\W =\End_{U(\ga)}
(Q_\chi)^{\text{op}}$, it is trivial to see that $Q_\chi$ is a
$(\ga, \W)$-bimodule and then $Q_\chi \otimes_{\W}V$ is a
$\ga$-module.

To check that $Q_\chi \otimes_{\W}V$ lies in $\whmod$, it suffices
to check that $a -\chi(a)$, $\forall a \in \m$  acts locally
nilpotently on $Q_\chi$ by the definition of the $U(\ga)$-action
on $Q_\chi \otimes_{\W}V$. Since $\m$ is negatively graded with
respect to the grading $\Gamma$, $\ad a =\ad (a -\chi(a))$ for any
$a \in \m$ acts locally nilpotently on $\ga$ and so locally
nilpotently on $U(\ga)$ (and also on $Q_\chi =U(\ga)/I_\chi$) by
induction on the PBW filtration length for $U(\ga)$.
\end{proof}

Let $\Wmod$ be the category of finitely generated $\mc
W_\chi$-modules. We define the {\em Whittaker functor}
\begin{equation*}
\wh:  \whmod \longrightarrow \Wmod, \qquad E \mapsto \wh (E).
\end{equation*}
We define another functor
$$
Q_\chi\otimes_{\W} -:   \Wmod   \longrightarrow \whmod, \qquad V
\mapsto Q_\chi\otimes_{\W}V.
$$

\subsection{The Skryabin equivalence}
\label{W:subsec:Skr=}

The following theorem is due to Skryabin \cite{Sk}, and here we
follow the new proof of \cite{GG}.

\begin{theorem}  \label{W:th:Skr=}
The functor $Q_\chi \otimes_{\W} -:   \Wmod   \longrightarrow
\whmod$ is an equivalence of categories, with $\wh: \whmod
\longrightarrow \Wmod$ as its quasi-inverse.
\end{theorem}

\begin{proof}

(1) We first prove that $\wh (Q_\chi\otimes_{\W}V) =V$.

Assume $V$ is generated as a $\W$-module by a finite-dimensional
subspace $V_0$. This gives rise to a $\W$-filtered module
structure on $V$ by letting $F_n V =(F_n \W). V_0$. Note that
$H^0(\m, Q_\chi\otimes_{\W}V) =\wh (Q_\chi\otimes_{\W}V)$, where
we regard $Q_\chi\otimes_{\W}V$ as an $\m$-module  with
$\chi$-twisted action.

We shall establish the following Claim in Lie algebra cohomology;
this readily implies that $H^0(\m, Q_\chi\otimes_{\W}V) \cong V$,
which is equivalent to $\wh (Q_\chi\otimes_{\W}V) =V$.

{\bf Claim~1}.  $\gr H^0(\m, (Q_\chi\otimes_{\W}V))
\stackrel{\mu_1}{\simeq} H^0(\m, \gr (Q_\chi\otimes_{\W}V))
\stackrel{\mu_2}{\simeq}  \gr V$; moreover,    $ H^i(\m,
(Q_\chi\otimes_{\W}V)) =H^i(\m, \gr (Q_\chi\otimes_{\W}V))=0$ for
$i>0$.

The two isomorphisms $\mu_1, \mu_2$ in the Claim are parallel to
the two isomorphisms $\nu_1, \nu_2$ in Theorem~\ref{W:th:2steps},
and they will be proved in the same strategy as in the proof of
Theorem~\ref{W:th:2steps}.

By Lemma~\ref{W:isom} and the diagram \eqref{W:QSlodowy}, we have
\begin{align*}
\gr (Q_\chi\otimes_{\W}V) &\cong \gr  Q_\chi\otimes_{\gr \W}\gr V
\\
 & \cong (\C[\texttt{M}'] \otimes {\gr \W})  \otimes_{\gr \W} \gr V
\cong  \C[\texttt{M}'] \otimes  \gr V.
\end{align*}
The isomorphism $\mu_2$ follows quickly from this.

The isomorphism $\mu_1$ follows exactly as the argument for the
isomorphism $\nu_1$ in the proof of Theorem~\ref{W:th:2steps}, once
we note that $Q_\chi\otimes_{\W}V$ is filtered with no
negative-degree components and $\m^*$ is $\N$-graded, and then
apply a spectral sequence argument.

(2) We shall show that, for any $E \in \whmod$, the canonical map
$$
\gamma: Q_\chi\otimes_{\W} \wh (E) \longrightarrow E, \quad
\bar{y} \otimes v \mapsto y.v
$$
is an isomorphism, where $\bar{y} \in Q_\chi =U(\ga)/I_\chi$ is
associated to $y \in U(\ga)$.

We first note that $\gamma$ is well-defined since $Q_\lag =U(\ga)
/I_\chi$ and $I_\chi$ is generated by $a-\chi(a), \forall a \in
\m$. Also observe that $\wh (E) =0$ implies that $E=0$.

Take an exact sequence of $\ga$-modules:
\begin{equation} \label{W:exact}
0 \longrightarrow E' \longrightarrow Q_\chi \otimes_{\W} \wh(E)
\stackrel{\gamma}{\longrightarrow} E \longrightarrow E''
\longrightarrow 0.
\end{equation}

Let us show that $\gamma$ is injective (i.e. $E'=0$). Indeed,
$$
\wh (E') =E' \cap \wh(Q_\chi \otimes_{\W} \wh(E))
 = E' \cap \wh(E) =0
$$
where the second equality follows from Part~ (1) above and the third equality
follows by definition of $\gamma$. Hence $E' =0$.

Now we prove that $\gamma$ is surjective (i.e. $E''=0$).
\eqref{W:exact} reduces now to a short exact sequence
$$0 \longrightarrow  Q_\chi \otimes_{\W} \wh(E)
\stackrel{\gamma}{\longrightarrow} E \longrightarrow E''
\longrightarrow 0
$$
which gives rise to a long exact sequence
$$
0 \rightarrow  H^0(\m, Q_\chi \otimes_{\W} \wh(E))
\stackrel{\gamma^*}{\rightarrow} H^0(\m, E) \rightarrow H^0(\m,
E'') \rightarrow H^1(\m, Q_\chi \otimes_{\W} \wh(E))
$$
where $H^1\bigl(\m, Q_\chi \otimes_{\W} \wh(E)\bigr)=0$ by
Claim~1. Then, this exact sequence can be rewritten as
$$
0 \longrightarrow  \wh (Q_\chi \otimes_{\W} \wh(E))
\stackrel{\gamma^*}{\longrightarrow} \wh(E) \longrightarrow
\wh(E'') \longrightarrow 0
$$
By (1), $\wh (Q_\chi \otimes_{\W} \wh(E)) =\wh(E)$ and $\gamma^*$
is an isomorphism. So, $\wh (E'') =0$, whence $E''=0$.
\end{proof}
\section{Good $\Z$-gradings in type $A$}
\label{W:sec:A}

In this section, we will describe the classification of good
$\Z$-gradings on $\ga =\gl_N$ or $\sgl_N$, due to Elashvili and
Kac \cite{EK}.

\subsection{Pyramids of shape $\la$}
\label{W:pyramid}

Given a partition $\la =(\la_1, \la_2, \ldots)$ of $N$, we
construct a combinatorial object, called pyramids (of shape
$\la$). We will illustrate the process by building pyramids of
shape $\la =(3,2,2)$.

We start with a (first) row of $\la_1=3$ boxes of size $2$ units
by $2$ units, with column numbers $1-\la_1, 3-\la_1, \cdots \la_1
-1$ (which is $-2, 0, 2$ in this example). We mark by $\bullet$ to
indicate the column number $0$.
$$
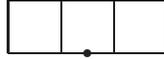
\begin{picture}(220,20)
\put(80,0){\line(1,0){60}} \put(80,20){\line(1,0){60}}
\put(140,0){\line(0,1){20}} \put(80,0){\line(0,1){20}}
\put(100,0){\line(0,1){20}} \put(120,0){\line(0,1){20}}
 \put(110,0){\circle*{3}}
\end{picture}
$$
Then, we add a (second) row of $\la_2=2$ boxes on top of the row
$1$. The rule is: keep the stair shape with permissible shifts by
integer units. In this example, we have three possibilities, and
this gives rise to $3$ pyramids of shape $(3,2)$.
$$
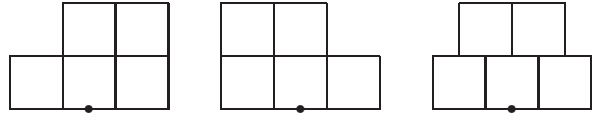
\begin{picture}(220,40)
\put(0,0){\line(1,0){60}} \put(0,20){\line(1,0){60}}
\put(20,40){\line(1,0){40}} \put(60,0){\line(0,1){40}}
\put(0,0){\line(0,1){20}} \put(20,0){\line(0,1){40}}
\put(40,0){\line(0,1){40}}
 \put(30,0){\circle*{3}}

\put(80,0){\line(1,0){60}} \put(80,20){\line(1,0){60}}
\put(80,40){\line(1,0){40}} \put(140,0){\line(0,1){20}}
\put(80,0){\line(0,1){40}} \put(100,0){\line(0,1){40}}
\put(120,0){\line(0,1){40}}
 \put(110,0){\circle*{3}}

\put(160,0){\line(1,0){60}} \put(160,20){\line(1,0){60}}
\put(170,40){\line(1,0){40}} \put(220,0){\line(0,1){20}}
\put(160,0){\line(0,1){20}} \put(180,0){\line(0,1){20}}
\put(200,0){\line(0,1){20}} \put(170,20){\line(0,1){20}}
\put(190,20){\line(0,1){20}} \put(210,20){\line(0,1){20}}
\put(190,0){\circle*{3}}
\end{picture}
$$

Then we repeat the process with the same rule by adding now a
(third) row of $\la_3=2$ boxes on top of row $2$. In this example,
we have only one permissible way of doing so, and so obtain three
pyramids of shape $(3,2,2)$ below (where the column numbers are
also indicated).
$$
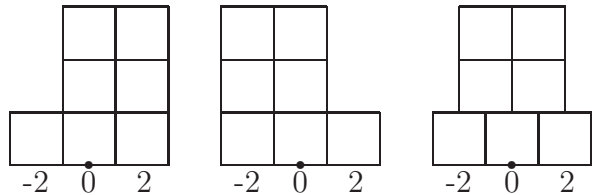
\begin{picture}(220,70)

\put(0,10){\line(1,0){60}} \put(0,30){\line(1,0){60}}
\put(20,50){\line(1,0){40}} \put(20,70){\line(1,0){40}}
\put(60,10){\line(0,1){60}} \put(0,10){\line(0,1){20}}
\put(20,10){\line(0,1){40}} \put(40,10){\line(0,1){60}}
\put(20,50){\line(0,1){20}}
 \put(30,10){\circle*{3}}
\put(5,0){-2} \put(27,0){0} \put(48,0){2}

\put(80,10){\line(1,0){60}} \put(80,30){\line(1,0){60}}
\put(80,50){\line(1,0){40}} \put(80,70){\line(1,0){40}}
\put(140,10){\line(0,1){20}} \put(80,10){\line(0,1){60}}
\put(100,10){\line(0,1){60}} \put(120,10){\line(0,1){60}}
 \put(110,10){\circle*{3}}
\put(85,0){-2} \put(107,0){0} \put(128,0){2}

\put(160,10){\line(1,0){60}} \put(160,30){\line(1,0){60}}
\put(170,50){\line(1,0){40}} \put(170,70){\line(1,0){40}}
\put(220,10){\line(0,1){20}} \put(160,10){\line(0,1){20}}
\put(180,10){\line(0,1){20}} \put(200,10){\line(0,1){20}}
\put(170,30){\line(0,1){40}} \put(190,30){\line(0,1){40}}
\put(210,30){\line(0,1){40}} \put(190,10){\circle*{3}}
\put(165,0){-2} \put(187,0){0} \put(208,0){2}
\end{picture}
$$

\begin{xca}   \label{W:xca:pyram}
There are $7$ pyramids of shape $(4,1)$ as there are $7$
permissible way of putting one box on top of
$$
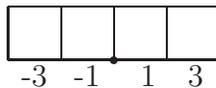
\begin{picture}(220,30)
\put(80,10){\line(1,0){80}} \put(80,30){\line(1,0){80}}
\put(80,10){\line(0,1){20}} \put(100,10){\line(0,1){20}}
\put(120,10){\line(0,1){20}} \put(140,10){\line(0,1){20}}
\put(160,10){\line(0,1){20}}
 \put(120,10){\circle*{3}}
 \put(85,0){-3} \put(105,0){-1} \put(130,0){1} \put(148,0){3}
\end{picture}
$$
\end{xca}

\subsection{}
\label{W:subsec:bij}

Given a pyramid $P$ of shape $\la$, let us fix a labeling by
numbers $\{1,2,\ldots, N\}$  of the $N$ boxes in $P$. A convenient
choice is to label downward from left to right in an increasing
order.

Let us take the second pyramid of shape $(3,2,2)$ as an example
with $N=7$. Our labeled pyramid reads
$$
\begin{picture}(220,70)

\put(80,10){\line(1,0){60}} \put(80,30){\line(1,0){60}}
\put(80,50){\line(1,0){40}} \put(80,70){\line(1,0){40}}
\put(140,10){\line(0,1){20}} \put(80,10){\line(0,1){60}}
\put(100,10){\line(0,1){60}} \put(120,10){\line(0,1){60}}
 \put(110,10){\circle*{3}}
 \put(90,20){\makebox(0,0){{3}}}
 \put(110,20){\makebox(0,0){{6}}}
 \put(130,20){\makebox(0,0){{7}}}
 \put(90,40){\makebox(0,0){{2}}}
 \put(110,40){\makebox(0,0){{5}}}
  \put(90,60){\makebox(0,0){{1}}}
\put(110,60){\makebox(0,0){{4}}}
\put(85,0){-2} \put(107,0){0} \put(128,0){2}
\end{picture}
$$

We fix a standard basis $v_i (1\le i\le N)$ of $\C^N$ associated
to the above labelling. Let
$$e =e^P =E_{14} +E_{25} +E_{36} +E_{67}
$$
be the nilpotent element in $\ga$ which sends a vector $v_i$ to
$v_{\texttt R(i)}$ where $\texttt R(i)$ denotes the label to the
right of $i$ in the labelled pyramid $P$ (by convention
$v_{\texttt R(i)} =0$ whenever $\texttt R(i)$ is not defined).
Note that $e$ has $J_\la$ as its Jordan form.

A $\Z$-grading $\Gamma^P$ of $\ga$ is determined by letting the
degree of a root vector for a simple root $\alpha_i =\varepsilon_i
-\varepsilon_{i+1}$ to be $\col_{i+1}  -\col_i$, where $\col_i$
denotes the column number of the box labelled by $i$ in $P$. In
other words, the grading operator is $h^P =h^P_\partial = -
\text{diag} \, (\col_1, \col_2, \ldots, \col_N)$ (note that $h^P$
is unique up to a shift of a scalar multiple of the identity
matrix $I_N$, and there is a unique shift which results to a
traceless matrix).

The grading $\Gamma$ is even if and only if a pyramid is {\em
even} in the sense that every box in $P$ lies immediately on top
of at most one box. For example, for $\la =(3,2,2)$ above, the
first and second pyramids are even while the third one is not.

\begin{example}  \label{W:ex:gl3}
Let $\ga =\gl_3$ and $e =E_{13}.$ Then the good grading for $e$ in
Example~\ref{W:ex:sl3}~(2) corresponds to the first pyramid and  the
Dynkin grading for $e$ corresponds to the third pyramid below:

$$
\begin{picture}(220,40)
\put(0,0){\line(1,0){40}} \put(0,20){\line(1,0){40}}
\put(0,40){\line(1,0){20}} \put(0,0){\line(0,1){40}}
\put(20,0){\line(0,1){40}} \put(40,0){\line(0,1){20}}
 \put(20,0){\circle*{3}}
 \put(10,10){\makebox(0,0){{2}}}
 \put(30,10){\makebox(0,0){{3}}}
 \put(10,30){\makebox(0,0){{1}}}

\put(80,0){\line(1,0){40}} \put(80,20){\line(1,0){40}}
\put(100,40){\line(1,0){20}} \put(80,0){\line(0,1){20}}
\put(100,0){\line(0,1){40}} \put(120,0){\line(0,1){40}}
 \put(100,0){\circle*{3}}
 \put(90,10){\makebox(0,0){{1}}}
 \put(110,10){\makebox(0,0){{3}}}
 \put(110,30){\makebox(0,0){{2}}}

\put(160,0){\line(1,0){40}} \put(160,20){\line(1,0){40}}
\put(170,40){\line(1,0){20}} \put(160,0){\line(0,1){20}}
\put(180,0){\line(0,1){20}} \put(200,0){\line(0,1){20}}
\put(170,20){\line(0,1){20}} \put(190,20){\line(0,1){20}}
\put(180,0){\circle*{3}}
 \put(170,10){\makebox(0,0){{1}}}
 \put(190,10){\makebox(0,0){{3}}}
 \put(180,30){\makebox(0,0){{2}}}
\end{picture}
$$
\end{example}

\begin{xca} \label{W:xca:gl3good}
Describe explicitly the good $\Z$-grading
corresponding to the second pyramid.
\end{xca}

\begin{theorem}  \label{W:th:EK} \cite{EK}
Let $\ga =\gl_N$ or $\sgl_N$, and let $\la$ be a partition of $N$.
There exists a one-to-one correspondence between the set of
pyramids of shape $\la$ and the set of good $\Z$-gradings for a
nilpotent matrix of Jordan shape $\la$ up to $GL_N$-conjugation,
$$
\{\text{Pyramids of shape } \la\}  \longrightarrow \{\text{good
}\Z\text{-gradings for a nilpotent matrix in }\mc O_\la\}/GL_N
$$
by sending $P$ to $\Gamma^P$.
\end{theorem}

\begin{proof}
It is easy to see the map is well-defined: different choices of
labelling of $P$ leads to another nilpotent matrix in the same
nilpotent orbit.

To see that $\Gamma^P$ is a good $\Z$-grading for $e =e^P$, we
will actually construct a $\N$-graded homogeneous basis for the
centralizer $\ga_e$ explicitly. Let $\ell =\la_1'$. We will choose
to label the first box of  row $1,\ldots, \ell$ of $P$ upward by
$1, \ldots, \ell$. Note that $\{e^{k_i} v_i \mid 1\le i \le \ell,
0\le k_i \le \la_i-1\}$ is a linear basis of $\C^N$, and indeed
each such $e^{k_i} v_i$ is equal to $v_j$ for some $1 \le j \le N$
which appears in the same row as $i$ in $P$.

An element $z$ in $\ga_e$ is determined by the values $z(v_i)$ for
$ 1\le i \le \ell$, since $z(e^k v_i) =e^k z(v_i)$. Consider for
now $z_{j,i;k_j} \in \ga_e  (1\le i,j \le \ell)$ such that
$z_{j,i;k_j}(v_i) =e^{k_j} v_j$ and  $z_{j,i;k_j}(v_{i'}) =0$ for
$1\le i' \le \ell, i' \neq i$. Then since $e^{\la_i +k_j} v_j
=z_{j,i;k_j}(e^{\la_i} v_i) =0$ and recall  $0\le k_j < \la_j$, we
see that the $k_j$'s have to satisfy the inequality:
$$
\la_j >k_j \ge \max(0, \la_j -\la_i), \quad \forall i,j
$$
and this condition is sufficient for $z_{j,i;k_j}\in \ga_e$ to be
well-defined.

Since $\min(\la_i, \la_j) =\la_j -\max(0, \la_j -\la_i)$, there
are in total $\sum_{1\le i,j \le \ell} \min(\la_i,\la_j)$ of such
$z_{j,i;k_j} \in \ga_e$. These elements $z_{j,i;k_j}$'s  in
$\ga_e$ are manifestly homogeneous, $\N$-graded and linearly
independent. Hence, they must form a basis for $\ga_e$, thanks to
the well-known identity $\dim \ga_e =\sum_{1\le i,j \le \ell}
\min(\la_i,\la_j)$ (which is also equal to $\sum_{i\ge 1} \la_i'$
in terms of the conjugate partition $\la'$).

Now we shall construct an inverse map. Let $e$ be a nilpotent
element in the nilpotent orbit $\mc O_\la$ and let $\Gamma =\ad h$
be a good $\Z$-grading for $e$. Up to $GL_N$-conjugation, we can
assume that $e =J_\la$, the canonical Jordan canonical form, and
that $h$ is a diagonal matrix. We can suitably relabel the
standard basis vectors $v_i$'s of $\C^N$ (or by a conjugation by
some permutation matrix in $GL_N$), so that the standard basis of
$\C^N$ consists of $\{e^{k_i} v_i \mid 1\le i \le \ell, 0\le k_i
\le \la_i-1\}$. In particular, $E_{ji}$ for all $i,j$ are
homogeneous with respect to the grading $\Gamma$. Now we visualize
a labeled ``generalized" pyramid $P$ of shape $\la$ as follows
(here generalized means that the staircase conditions on both
sides of $P$ are not verified at the moment): the row one of $P$
is fixed with column numbers $1-\la_1, 3-\la_1, \cdots \la_1 -1$
with the leftmost box labelled by $1$, and the remaining rows are
fixed by letting the leftmost box of row $j$ to have a column
number $\deg E_{j1} +1-\la_1$.

Take the basis $z_{j,i;k_j}$ for $\ga_e$ constructed earlier. The
good grading $\Gamma$ implies that $z_{j,i;0} =E_{j i} +\ldots$
have non-negative integral degrees, and so do $E_{j i}$, for
$1\leq j<i \leq \ell$. This ensures that the left-hand side of the
generalized pyramid $P$ is of staircase shape. Similarly, the
consideration of $z_{j,i;\la_j -\la_i}$ for $1\leq j<i \leq \ell$
shows the right-hand side of $P$ is of staircase shape.

Hence, $P$ is indeed a pyramid. It follows by the construction of
$P$ that $e =e^P$.
\end{proof}

\begin{corollary}   \label{W:cor:even}
There exists a bijection between even pyramids of size $N$ and
even good $\Z$-gradings of $\ga$ up to $GL_N$-conjugation (by
sending $P$ to $\Gamma^P$).
\end{corollary}
\begin{proof}
Follows from the way a grading is defined using a pyramid.
\end{proof}

\begin{corollary} \label{W:cor:even exist}
Given a nilpotent element $e \in \ga$, there exists an even good
$\Z$-grading of $\ga$ for $e$.
\end{corollary}
\begin{proof}
Assume the Jordan form of $e$ corresponds to a partition $\la$.
The Young diagram  $\la$ in French fashion is an even pyramid of
shape $\la$.
\end{proof}

A pyramid $P$ is call {\em symmetric} if $P$ is self-dual with
respect to the reflection along the vertical line passing through
the mid-point of row one of $P$. For example, the third pyramid of
shape (3,2,2) in Section~\ref{W:pyramid} is symmetric.

\begin{corollary} \label{W:cor:Dynkin}
There exists a bijection from the set of symmetric pyramids of
size $N$ to the set of Dynkin gradings of $\ga$ up to
$GL_N$-conjugation (by sending $P$ to $\Gamma^P$).
\end{corollary}

\begin{proof}
Let $P$ be a symmetric pyramid of size $N$. Then the good
$\Z$-grading $\Gamma^P$ of $\ga$ is Dynkin, since we can take
$\sgl_2$ to be the diagonal subalgebra of the direct sum of the
$\sgl_2$'s associated to each row of $P$ (which corresponds to a
Jordan block) as constructed in Example~\ref{W:ex:regular}.

Any two Dynkin gradings for a given nilpotent element $e$ are
conjugated by $GL_N$, as we have seen in the course of
establishing the bijection $\Omega$ in \ref{W:bijection}.
\end{proof}

\begin{remark} \label{W:othertype}
The classification of good $\Z$-gradings for other types
was also carried out in \cite{EK}. See Baur and Wallach \cite{BW}
for closely related classification of nice parabolic subalgebras and
further clarification.
\end{remark}

\section{$W$-algebras and independence of good gradings}
\label{W:sec:indep good}

In this section, we shall prove that the isoclasses of finite
$W$-algebras are independent of the choices of the good
$\R$-gradings $\Gamma$, following Brundan and Goodwin \cite{BG}.

\subsection{}
\label{W:subsec:Rgrade} In this subsection, we fix notations which
will be used throughout Section~\ref{W:sec:indep good}.

Let $e \in \ga$ be a nilpotent element. Fix an $\sgl_2$-triple
$\{e,h,f\}$ in $\ga$, and denote this copy of $\sgl_2$ by $\mf s$.

Recall that the $\ad h$-eigenspace decomposition of $\ga$ defines
the Dynkin $\Z$-grading
$$
\Gamma^\emptyset: \ga = \bigoplus_{j \in \Z} \ga_j ,
 \quad \text{ where } \ga_0
=\ga_h.
$$
Let $\c = \ga_h$ and let $C$ be the corresponding closed connected
subgroup of the adjoint group $G$. Also let $\mf r = \bigoplus_{j
> 0} \ga_j$ (associated to $\Gamma^\emptyset$) and let $R$ be the
corresponding closed connected subgroup of $G$. It is well known
that $C_e$ is a maximal reductive subgroup of $G_e$, with Lie
algebra $\c_e$, and that $R_e$ is the unipotent radical of $G_e$,
with Lie algebra $\mf r_e$.

Fix a maximal torus $T$ of $G$ contained in $C$ and containing a
maximal torus of $C_e$. An important role is played by the
centralizer $\t_e$ of $e$ in the Lie algebra $\t$ of $T$. It is a
Cartan subalgebra of the reductive part $\c_e$ of the centralizer
$\ga_e$.


%
%
\subsection{} \label{W:subsec:Rgrading}

An $\R$-grading
$$
\Gamma : \ga = \bigoplus_{j \in \R} \ga_j
$$
of $\ga$ is called a {\em good grading} (or {\em good
$\R$-grading}) for $e$ if the conditions
(\ref{W:deg=2})-(\ref{W:surj}) in \ref{W:Dynkin gr}
holds (where the indices are taken
in $\R$). Lemma~\ref{W:inj=sur} (i.e. $\eqref{W:inj}
\Leftrightarrow \eqref{W:surj}$) remains to be valid for
$\R$-gradings.

We can associate to any $\R$-grading a semisimple element
$h_\Gamma$, as in Lemma~\ref{W:grading}.
\begin{definition}   \label{W:def:compat}
We say that a grading $\Gamma: \ga = \bigoplus_{j \in \R} \ga_j$
is {\em compatible with $\t$} if $\t \subseteq \ga_0$, or
equivalently, if $h_\Gamma \in \t$.
\end{definition}

\begin{remark}  \label{W:h in t}
We have $h =h_{\Gamma^\emptyset} \in \t$ since $\t \subseteq \c$
implies $[h, \t]=0$ and $\t$ is maximal abelian. Hence, the Dynkin
grading $\Gamma^\emptyset$ is compatible with $\t$.
\end{remark}

\begin{lemma}  \label{W:compatExist}
Every $\R$-grading $\Gamma$ is $G$-conjugate to a $\R$-grading
that is compatible with $\t$.
\end{lemma}

\begin{proof}
Every semisimple element of $\ga$, in particular $h_\Gamma$, is
$G$-conjugate to an element of $\t$. The lemma follows.
\end{proof}

Let $E_e$ be the $\R$-form for $\t_e$ consisting of all $p \in
\t_e$ such that the eigenvalues of $\ad p$ on $\ga$ are real.
\begin{lemma}\label{W:choice}
\begin{enumerate}
\item Any good $\R$-grading for $e$ is $G_e$-conjugate to a good
$\R$-grading $\Gamma$ for $e$ compatible with $\t$.

\item
For such a $\Gamma$, we have that $h \in \ga_0, f \in \ga_{-2}$,
and $h_\Gamma = h + p$ for some point $p \in E_e$.
\end{enumerate}
\end{lemma}

\begin{proof}
(1) Recall that any two $\sgl_2$-triples in $\ga$ sharing the same
$e$ are conjugate by $G_e$, as seen from the proof of
Theorem~\ref{W:th:omega}. By Lemma~\ref{W:lem:EK}, we may assume that
the $\sgl_2$-triple $\mf s =\{e,h,f\}$ we start with is
$\Gamma^\emptyset$-graded, that is, $h \in \ga_0$ and $f \in
\ga_{-2}$.

Take the semisimple element $h_{\Gamma^\emptyset} \in \ga_0
=\ga_h$ as in Lemma~\ref{W:grading}. Let $p' =h_{\Gamma^\emptyset}
-h$, which is a semisimple element lying in $\c_e$ (which is the
centralizer of $h$ and $e$), since $[p',e]=[h_{\Gamma^\emptyset},
e] -[h,e] =2e-2e =0$. Then $p'$ is $C_e$-conjugate to an element
in $\t_e$, a Cartan subalgebra of $\c_e$. Write $\text{Ad} g (p')
=p \in \t_e$ for $g \in C_e \subseteq G_e$. Recall $h \in \t$ from
Remark~\ref{W:h in t}. Hence, $\Ad g(h_{\Gamma^\emptyset}) =\Ad
g(h + p') =h + p \in \t$, and the good grading $\Gamma :=\Ad g
({\Gamma^\emptyset})$ for $e$ is compatible with $\t$.

(2) For such a $\Gamma$ as in (1), we have $h \in \ga_0$ since $\t
\subseteq \ga_0$ and $h \in \t$. Then $p=h_\Gamma -h$ centralizes
$e$ and $h$, hence also $f$, so $[h_\Gamma,f]=[h,f]=-2f$ and $f
\in \ga_{-2}$. Finally, $p$ belongs to $\t_e$, hence to $E_e$,
since $\Gamma$ is an $\R$-grading.
\end{proof}
\subsection{}
\label{W:subsec:polytope}

For $p \in E_e$, we let $\Gamma^p$ denote the $\R$-grading of
$\ga$ defined by $\ad (h+p)$ (which is consistent with the
notation for the Dynkin grading $\Gamma^\emptyset$). Such a
grading $\Gamma^p$ is automatically compatible with $\t$, but it
may not be a good grading.

Note that $\t_e$ commutes with $h$ and $e$, and hence with $\s$.
For $\alpha \in \t_e^*$ and $i \geq 0$, let $L(\alpha,i)$ denote
the irreducible $\t_e \oplus \s$-module of dimension $(i+1)$ on
which $\t_e$ acts by weight $\alpha$. Decompose $\ga$ as an
$\t_e\oplus \s $-module
$$
\ga \cong \bigoplus_{\alpha \in \Phi_e \cup\{0\}} \bigoplus_{i
\geq 0} m(\alpha,i) L(\alpha,i)
$$
where $\Phi_e$ denotes the set of nonzero weights of $\t_e$ (or
$E_e$) of $\ga$.
One can show that $\Phi_e$ is a so-called restricted root system
on $E_e$ (cf. \cite{BG}), but for our purpose we only need to
record two simple facts: $\alpha \in \Phi_e$ implies $\pm \alpha
\in \Phi_e$ and $\Phi_e$ spans $E_e^*$.
For $\alpha \in \Phi_e$, let
$$
d(\alpha) = 1+\min\{i \geq 0 \mid m(\alpha,i)
\neq 0\}
$$
be the minimal dimension of a simple $\mf s$-submodule of the
$\alpha$-weight space of $\ga$.

\begin{proposition}\label{W:ggp}
Let $p \in E_e$. The grading $\Gamma^p$ is a good grading for $e$
if and only if $|\alpha(p)| < d(\alpha)$ for all $\alpha \in
\Phi_e$.
\end{proposition}

\begin{proof}
Recall that $\Gamma^p:\ga = \bigoplus_{j \in \Z} \ga_j$ is a good
grading for $e$ if and only if $\ga_e \subseteq \bigoplus_{j > -1}
\ga_j$. Note that $h+p$ acts on the highest weight vector of
$L(\alpha,i)$ as the scalar $\alpha(p)+i$ for $\alpha \in \Phi_e$.
So, $\Gamma^p$ is a good grading for $e$ if and only if
$\alpha(p)+i > -1$ for all $\alpha \in \Phi_e$; the latter (by
considering $\pm \alpha \in \Phi_e$ simultaneously) is equivalent
to $|\alpha(p)| < 1+i$ for all $\alpha \in \Phi_e$.
\end{proof}

The {\em good grading polytope} for $e$ is defined to be
$$\mathscr P_e
 =\{p \in E_e \mid \Gamma^p \text{ is a good } \R\text{-grading for }e\}.
$$
Since $\Phi_e$ spans $E_e^*$,  $\mathscr P_e$ is an open convex
polytope in $E_e$ by Proposition~\ref{W:ggp}: indeed, $|\alpha (t
p_1 +(1-t) p_2)| \leq t|\alpha (p_1)| +(1-t) |\alpha (p_2)| <
d(\alpha)$, for $p_1, p_2 \in \mathscr P_e$ and $0\le t \le 1$.

\subsection{}
\label{W:subsec:adj}

Two good $\R$-gradings $\Gamma:\ga = \bigoplus_{i \in \R} \ga_i$
and $\Gamma':\ga = \bigoplus_{j \in \R} \ga_j'$ are {\em adjacent}
if
$$
\ga = \bigoplus_{i^- \leq j \leq i^+} \ga_i \cap \ga_j',
$$
where $i^-$ denotes the largest integer strictly smaller than $i$
and $i^+$ denotes the smallest integer strictly greater than $i$.

\begin{lemma} \label{W:sameinterval}
The following conditions are equivalent for $i,j \in \R$:
\begin{enumerate}
\item $i^- \leq j \leq i^+$;

\item $i$ and $j$ lie in the same unit interval $[a, a+1]$ for
some $a\in \Z$;

\item $j^- \leq i \leq j^+$.
\end{enumerate}
\end{lemma}

\begin{xca} \label{W:xca:adj}
Prove Lemma~\ref{W:sameinterval}.
\end{xca}

The next lemma allows us to reduce the study of adjacent good
$\R$-gradings to the good grading polytope $\mathscr P_e$.

\begin{lemma}\label{W:doublegr}
Every pair of adjacent good $\R$-gradings for $e$ is
$G_e$-conjugate to a pair of adjacent good $\R$-gradings for $e$
compatible with $\t$.
\end{lemma}

\begin{proof}[Sketch of a proof]
For adjacent gradings $\Gamma$ and $\Gamma'$, $\ad h_\Gamma$ and
$\ad h_{\Gamma'}$ are by definition simultaneously diagonalizable,
and hence $[h_\Gamma, h_{\Gamma'}] =0$.

We can refine Lemma~\ref{W:lem:EK} with essentially the same proof
to find an $\sgl_2$-triple $\mf s =\{e,h,f\}$ which is doubly
$\Gamma$-graded and $\Gamma'$-graded, that is, $h \in \ga_0 \cap
\ga_0'$ and $f \in \ga_{-2} \cap \ga_{-2}'$.

The lemma can now be viewed as a ``double" analogue of
Lemma~\ref{W:choice} ~(1), and it can be proved by refining the
arguments therein with the help of the doubly graded
$\sgl_2$-triple above.
\end{proof}

\begin{xca} \label{W:xca:double}
Fill in the details of the proof of Lemma~\ref{W:doublegr}.
\end{xca}

Consider the hyperplanes
$$
 H_{\alpha, k} =\{p \in E_e \mid \alpha (p) =k \},
 \quad \alpha \in \Phi_e, k \in \Z.
$$
The connected components of $E_e \backslash \bigcup_{\alpha, k}
H_{\alpha, k}$ will be referred to as (open) {\em alcoves}.

\begin{lemma}\label{W:adj}
Let $p,p' \in \mathscr P_e$. Then, $\Gamma^p$ and $\Gamma^{p'}$
are adjacent if and only if $p$ and $p'$ belong to the closure of
the same alcove.
\end{lemma}

\begin{proof}
By definition, $\ga_\alpha (\forall \alpha \in \Phi)$ lies in the
degree $\alpha(h + p)$ subspace of $\ga$ with respect to
$\Gamma^{p}$, and a similar remark applies to $\Gamma^{p'}$. Note
that $\alpha(h) \in \Z$. It follows that $\Gamma^{p}$ and
$\Gamma^{p'}$ are adjacent if and only if $\alpha(p)^- \leq
\alpha(p') \leq \alpha(p)^+, \forall \alpha \in \Phi$, and by
Lemma~\ref{W:sameinterval}~(2), if and only if $p$ and $p'$ belong
to the closure of the same alcove.
\end{proof}

\begin{proposition}  \label{W:n gradings}
Given any two good gradings $\Gamma$ and $\Gamma'$ for $e$, there
exists a chain $\Gamma  =\Gamma_0, \Gamma_1, \dots, \Gamma_n
=\Gamma'$ of good gradings for $e$ such that $\Gamma_{i}$ is
adjacent to $\Gamma_{i+1}$ for each $i=0,\ldots,n-1$.
\end{proposition}

\begin{proof}
A simultaneous conjugation preserves pairs of adjacent good
gradings. So by Lemma~\ref{W:doublegr}, we may assume that $\Gamma
=\Gamma^p$ and $\Gamma' =\Gamma^{q}$ are good gradings for $e$
compatible with $\t$, for some $p, q\in \mathscr P_e$. As a convex
polytope, $\overline{\mathscr P}_e$ is connected, and it follows
from Proposition~\ref{W:ggp} that the closure $\overline{\mathscr
P}_e$ is a union of the closures of finitely many alcoves. Now the
proposition follows from Lemma~\ref{W:adj}.
\end{proof}

\begin{lemma}\label{W:sameLag}
Let $\Gamma:\ga=\bigoplus_{i \in \R} \ga_i$ and
$\Gamma':\ga=\bigoplus_{j \in \R} \ga_j'$ be adjacent good
gradings for $e$. Then there exist Lagrangian subspaces $\lag$ of
$\ga_{-1}$ and $\lag'$ of $\ga_{-1}'$
such that
$$
\lag \bigoplus (\bigoplus_{i < -1} \ga_i) = \lag' \bigoplus
(\bigoplus_{j < -1} \ga_j').
$$
\end{lemma}

\begin{proof}
We have that
\begin{align*}
\ga_{-1} &= ((\ga_{-1}\cap \ga_{<-1}')\bigoplus (\ga_{-1} \cap
\ga_{>-1}'))
\bigoplus (\ga_{-1} \cap \ga_{-1}'),\\
\ga_{-1}' &= ((\ga_{<-1}\cap \ga_{-1}') \bigoplus (\ga_{>-1} \cap
\ga_{-1}')) \bigoplus (\ga_{-1} \cap \ga_{-1}').
\end{align*}
Note the last summand $\ga_{-1} \cap \ga_{-1}'$ is orthogonal to
the first two summands in either identity above with respect to
$\langle\cdot,\cdot\rangle =(e |[\cdot,\cdot])$. Hence
$\langle\cdot,\cdot\rangle$ regarded as a (skew-symmetric)
bilinear form on $\ga_{-1} \cap \ga_{-1}'$ remains to be
non-degenerate. Take a Lagrangian subspace $\mf k$ of $\ga_{-1}
\cap \ga_{-1}'$, and set $\lag = \mf k \oplus (\ga_{-1} \cap
\ga_{<-1}')$ and $\lag' = \mf k \oplus (\ga_{<-1} \cap
\ga_{-1}')$. The lemma follows by using the assumption of
adjacency to check the inclusion of subspaces $ \lag \oplus
(\oplus_{i < -1} \ga_i) \subseteq \lag' \oplus (\oplus_{j < -1}
\ga_j')$ and the opposite inclusion by symmetry.
\end{proof}

\begin{theorem}\label{W:th: W Gamma}  \cite{BG}
The finite $W$-algebras associated to any two good gradings
$\Gamma$ and $\Gamma'$ for $e$ are isomorphic.
\end{theorem}

\begin{proof}
In this proof, we shall write a finite $W$-algebra as $\mc
W_{\lag, \Gamma}$, indicating its dependence on a Lagrangian
$\lag$ and a good $\R$-grading $\Gamma$.

Fix a chain $\Gamma  =\Gamma_0, \Gamma_1, \dots, \Gamma_n
=\Gamma'$ of good gradings for $e$ such that $\Gamma_{i}$ is
adjacent to $\Gamma_{i+1}$ for each $i$, as in Proposition~\ref{W:n
gradings}.

By Lemma~\ref{W:sameLag}, we can suitably choose Lagrangians
$\lag_i$ and $\lag_i'$ with respect to the gradings $\Gamma_{i}$
and $\Gamma_{i+1}$, respectively, which result an equality of the
corresponding Lie subalgebras $\m_i =\m_{i+1}$. So by the
Whittaker model definition of $W$-algebras, we have $\mc
W_{\lag_i, \Gamma_i}=\mc W_{\lag_i', \Gamma_{i+1}}$. Accordingly
to Gan-Gunzburg's Theorem~\ref{W:th:W lag} (or rather its
generalization in the setting of $\R$-gradings), the algebras $\mc
W_{\lag, \Gamma}$ are isomorphic for a fixed grading $\Gamma$ and
different choices of isotropic subspaces $\lag \subseteq
\ga_{-1}.$ Hence, the theorem follows by composing a sequence of
algebra isomorphisms:
$$
\mc W_{\lag_1, \Gamma_1}=\mc W_{\lag_1', \Gamma_2} \cong \mc
W_{\lag_2, \Gamma_2}=\mc W_{\lag_2', \Gamma_{3}} \cong \cdots
\cong \mc W_{\lag_{n-1}, \Gamma_{n-1}}=\mc W_{\lag_{n-1}',
\Gamma'}.
$$

\end{proof}

\section{Higher level Schur duality}
\label{W:sec:duality}

In this section, we let $\ga =\gl_N$. We present a duality between
finite $W$-algebras and degenerate cyclotomic Hecke algebras, due
to Brundan and Kleshchev \cite{BK3}.
\subsection{Schur duality}
\label{W:subsec:schur}

Let $V =\C^N$. Then the tensor space $V^{\otimes d}$ is naturally
a $(U(\ga), S_d)$-bimodule, which will be expressed as
$$
U(\ga) \;  {\curvearrowright} \; V^{\otimes d} \;
 {\curvearrowleft} \; S_d.
$$
The celebrated Schur duality states that the images of $U(\ga)$
and $\C S_d$ form mutual centralizers in $\End (V^{\otimes d})$.
Since the $S_d$-module $V^{\otimes d}$ is semisimple, one further
obtains a multiplicity-free decomposition of $V^{\otimes d}$ as a
$(U(\ga), S_d)$-bimodule.

\subsection{A duality of graded algebras}
\label{W:subsec:vust}

Let $e \in \ga$ be nilpotent, and let $\chi \in \ga^*$ be
associated to $e$ as before by $(\cdot|\cdot)$. Assume the Jordan
canonical form of $e$ is given by a partition $\la$ of $N$. We
choose the (even) pyramid to be the Young diagram $\la$ in the
French fashion, labelled by $\{1, \ldots, N\}$ down the columns
from left to right.

Denote by $\ell =\la_1$. Then $e^\ell =0$.

Denote by $\overline{H}_d := \C [x_1, \ldots, x_d] \rtimes \C S_d$
the semi-direct product algebra formed by the natural action of
$S_d$ on the polynomial algebra $\C [x_1, \ldots, x_d]$. We denote
by $\C_\ell [x_1, \ldots, x_d]$ the truncated polynomial algebra
$\C[x_1, \ldots, x_d]/\langle x_1^\ell, \ldots, x_n^\ell \rangle$,
and then the algebra $\overline{H}_d^\ell := \C_\ell [x_1, \ldots,
x_d] \rtimes \C S_d$ is a quotient of $\overline{H}_d$. Both
$\overline{H}_d$ and $\overline{H}_d^\ell$ are $\N$-graded
algebras with $\deg \sigma =0, \forall \sigma \in S_n$, and $\deg
x_i =2, \forall i.$

We define an $\overline{H}_d$-module structure on $V^{\otimes d}$
by letting
$$x_i =1^{\otimes (i-1)} \otimes e \otimes 1^{\otimes (d-i)},
\quad 1\le i \le d.$$ The action of $\overline{H}_d$ factors
through $\overline{H}^\ell_d$. Letting $\ga_e$ act as the
subalgebra of $\ga$, we clearly have a $(U(\ga_e),
\overline{H}^\ell_d)$-bimodule structure on $V^{\otimes d}$:
\begin{equation} \label{W:Vust}
U(\ga_e) \; \stackrel{\overline{\Phi}_{d}}{\curvearrowright} \;
V^{\otimes d} \; \stackrel{\overline{\Psi}_{d}}{\curvearrowleft}
\; \overline{H}^\ell_d.
\end{equation}

The Vust duality (cf. \cite[Theorem~2.4]{BK3}) states that the
images of $U(\ga_e)$ and $\overline{H}^\ell_d$ form mutual
centralizers in $\End (V^{\otimes d})$, i.e.
$$
\overline{\Phi}_{d}(U(\ga_e))=
\End_{\overline{H}^\ell_d}(V^{\otimes d}), \quad
\End_{U(\ga_e)}(V^{\otimes d})^{\operatorname{op}} =
\overline{\Psi}_d(\overline{H}^\ell_d).
$$
In general, these images are not semisimple algebras.

\subsection{Higher level Schur duality}
\label{W:subsec:higher}

Recall from Section~\ref{W:W=m inv} the definition of $\W$ as
$\m$-invariants in $U(\mf p)$, where $\mf p$ is given in
\eqref{W:def:p}. Let $\eta:U(\mathfrak{p}) \rightarrow
U(\mathfrak{p})$ be the algebra automorphism defined by
\begin{align*}
\eta(e_{i,j}) &= e_{i,j} + \delta_{i,j} (\la_1 -\la'_{\col(j)} -
\la'_{\col(j)+1} - \cdots - \la'_\ell), \quad e_{i,j} \in
\mathfrak{p},
\end{align*}
where $\la'$ denotes the conjugate partition of $\la$. For
simplicity of some later formulas, we will adopt the definition of
$\W$ by a twist of automorphism $\eta$ as follows:
\begin{align*}
\W &= \{y \in U(\mathfrak{p})\mid [a, \eta(y)] \in I_\chi, \forall
a \in \mathfrak{m}\}.
\end{align*}
Clearly, this new definition is isomorphic to the one in
Section~\ref{W:W=m inv}.

The higher level Schur duality is a filtered deformation of the
Vust duality \eqref{W:Vust} of the following form:
$$
\W \; \stackrel{\Phi_{d}}{\curvearrowright} \; V^{\otimes d} \;
\stackrel{\Psi_{d}}{\curvearrowleft} \; H_d^{[\la]}
$$
where $U(\ga_e)$ is replaced by the $W$-algebra $\W$ and
$\overline{H}_d^\ell$ by a cyclotomic Hecke algebra $H_d^{[\la]}$.
Recall that the degenerate affine Hecke algebra $H_d$ (introduced
by Drinfeld \cite{Dr}) is generated by its polynomial subalgebra
$\C [x_1, \ldots, x_d]$ and a subalgebra $\C S_d$, and that it
satisfies the additional relation for $1\le i,j \le d$:
$$s_i x_j = x_j s_i \ (j \neq i,i+1), \quad
s_i x_{i+1} = x_i s_i + 1.
$$
The cyclotomic Hecke algebra $H_d^{[\la]}$ is defined to be the
quotient of $H_d$ by the two-sided ideal generated by
$\prod_{i=1}^l (x_1 - \la_i' +\la_1)$.

The remainder of this section is to explain the two (commuting)
actions on $V^{\otimes d}$ of $\W$ and $H_d$; the action of $H_d$
factors through $H_d^{[\la]}$.
\subsection{A theorem of Arakawa and Suzuki} \label{W:AS}

Given a $\ga$-module $M$, we consider the endomorphism of
$\ga$-module $x: M \otimes V \rightarrow M \otimes V$, where $x$
acts by $\Omega =\sum_{1\le i,j\le N} E_{ij} \otimes E_{ji}.$ We
also let $s =1 \otimes \Omega: M \otimes V \otimes V \rightarrow M
\otimes V \otimes V$, which can be easily seen to coincide with
the permutation of the two copies of $V$. Abstractly, we may
regard $x$ as an endomorphism of the functor $-\otimes V$ and $s$
as an endomorphism of the functor $(-\otimes V)^2$. This provides
an example of $\sgl_2$-categorification of Chuang and Rouquier
\cite{CR}.

Now let $x_i = 1^{d-i} x 1^{i-1} (1\le i \le d)$ and $s_j
=1^{d-j-1}s1^{j-1} (1\le j \le d-1)$ be the endomorphisms of the
$d$th power $(-\otimes V)^d$. Equivalently, by the natural
isomorphism $(-\otimes V)^d \cong -\otimes V^{\otimes d}$, we have
the corresponding endomorphisms
$$
\hat{x}_i =\sum_{k=0}^{i-1} \Omega^{(k, i)} =\Omega^{(0,i)}
+\sum_{1\le k<i} s_{ki},
 \qquad \quad   \hat{s}_j = \Omega^{(j,j+1)}
 $$
on $M \otimes V^{\otimes d}$ for a $\ga$-module $M$, where
$s_{ki}$ denotes the permutation of the $k$th and $i$th copies of
$V$. Here we have labeled the tensor factors from $0$ to $d$, and
the notation $\Omega^{(a,b)}$ means $\sum_{ i,j} 1\otimes \cdots
E_{ij} \otimes \cdots  \otimes E_{ji} \cdots \otimes 1$ with the
two nontrivial factors at $a$th and $b$th places.

According to a theorem of Arakawa and Suzuki \cite{AS}, these
$\hat{x}_i$'s and $\hat{s}_j$'s satisfy the relations of the
degenerate affine Hecke algebra $H_d$. Hence we have obtained a
$(\ga, H_d)$-bimodule structure on $M \otimes V^{\otimes d}$ for
any $\ga$-module $M$.

\begin{xca} \label{W:xca:AS}
Prove that $\hat{x}_i$ for $1\le i \le d$ and $\hat{s}_j$ for $1\le j \le d-1$ satisfy the defining relations of the
degenerate affine Hecke algebra $H_d$.
\end{xca}

\subsection{Higher level Schur duality}
\label{W:subsec:higherdual}

Recall the Skryabin equivalence of categories from
Section~\ref{W:subsec:functors}: $ \wh:   \whmod \rightarrow
\Wmod$ and $Q_\chi\otimes_{\W} -:    \Wmod   \rightarrow \whmod.$
Via such an equivalence, we can transfer functors on the category
$\whmod$ to functors on $\Wmod$.

Given a finite-dimensional $\ga$-module $X$, the functor $-\otimes
X: \whmod \rightarrow \whmod$ gives rise to, via the Skryabin
equivalence, an exact functor
$$-\circledast X:\Wmod \longrightarrow \Wmod ,
\quad M \mapsto \wh((Q_\chi \otimes_{\mc W_\chi} M) \otimes X).
$$
These functors satisfy the naturality
$
(M \circledast X) \circledast Y
\stackrel{\sim}{\rightarrow} M \circledast (X \otimes Y)$ for
another finite dimensional $\ga$-module $Y$.
For our purpose, we will mainly consider the cases when $X$ and
$Y$ are tensor powers of $V$.

By using the functors $-\circledast V^{\otimes d}$ on $\Wmod$ in
place of the functors $-\otimes V^{\otimes d}$ on $\whmod$, we
obtain via the Skryabin equivalence, for any $\W$-module $M$, a
$(\W, H_d)$-bimodule structure on $M \circledast V^{\otimes d}$.
We will be mostly interested in the case when $M =\C$, the
restriction to $\W$ of the trivial $U(\mf p)$-module, but the
arguments below are equally valid for a $U(\mf p)$-module $M$.

Suppose that $M$ is actually a $\mathfrak{p}$-module, and view $M$
and $M \otimes X$ as $\W$-modules by restricting from
$U(\mathfrak{p})$ to $\W$. Consider the map
\begin{equation*}
\hat\mu_{M,X}:(Q_\chi \otimes_{\mc W_\chi} M) \otimes X
\rightarrow M \otimes X
\end{equation*}
which sends $(\eta(u) 1_\chi \otimes m) \otimes x \mapsto um
\otimes x$ for $u \in U(\mathfrak{p})$, $m \in M$ and $x \in X$.
The restriction of this map to the subspace $\wh((Q_\chi
\otimes_{\mc W_\chi} M) \otimes X)$ obviously defines a
{homomorphism} of $\W$-modules
\begin{equation*}
\mu_{M, X}: M \circledast X {\longrightarrow} M
\otimes X.
\end{equation*}
It turns out that this homomorphism is actually an isomorphism
\cite{BK3} (also cf. \cite{Go}). Roughly speaking, the isomorphism
can be obtained by using the following $\m$-module isomorphism of
Skryabin:
$$Q_\chi \otimes_{\mc W_\chi} M \cong M \otimes_\C
U(\m)^\star
$$
where $U(\m)^\star =\lim\limits_{\stackrel{\longleftarrow}{k}}  U(\m)
/I^\m_{\chi,k}$, and $I^\m_{\chi,k}$ is the left ideal of $U(\m)$
generated by $(a-\chi(a))^k, \forall a \in \m$.

In case when $M =\C$, we have obtained an isomorphism of $(\W,
H_d)$-bimodules
\begin{equation}   \label{W:eq:same}
\C \circledast V^{\otimes d} \stackrel{\cong}{\longrightarrow}
V^{\otimes d}. \end{equation}

Recall the standard basis $v_i, 1\le i \le N$ for $V =\C^N$. Then
$V^{\otimes d}$ has a basis $v_{\bf i} =v_{i_1}\otimes \ldots
\otimes v_{i_d}$, where ${\bf i} =(i_1, \ldots, i_d)$ with $1\le
i_1, \ldots, i_d \le N$. Denote by $\texttt{L}(i)$ the number in
the box immediately to the left  of the $i$th box in the pyramid
$\la$ if it exists.

\begin{lemma}\label{W:action}
The element $x_1 \in H_d$ sends each $v_{\bf i} \in V^{\otimes d}$
to
$$
v_{{\texttt{L}}{\bf i}} + (\la_{\col(i_1)}' -\la_1)
v_{\bf i} - \sum_{\substack{1 < k \leq d \\
\col(i_k) < \col(i_1)}} v_{{\bf i}\cdot(1,k)},
$$
where $v_{{\texttt{L}}{\bf i}} =v_{\texttt{L}(i_1)} \otimes
v_{i_2} \otimes \ldots \otimes v_{i_d}$ if $\texttt{L}(i_1)$ is
defined; otherwise, set $v_{{\texttt{L}}{\bf i}} =0$.
\end{lemma}
\begin{proof}[Sketch of a proof]
By examining carefully step by step how the
isomorphism~\eqref{W:eq:same} is obtained, the action of $x_1$ is
traced back to the corresponding action in Section~\ref{W:AS}.
\end{proof}

\begin{theorem}  \cite{BK3} \label{W:dcthm}
The action of $H_d$ on $V^{\otimes d}$ factors through
$H_d^{[\la]}$ and this gives rise to commuting actions:
$$
\W \; \stackrel{\Phi_{d}}{\curvearrowright} \; V^{\otimes d} \;
\stackrel{\Psi_{d}}{\curvearrowleft} \; H_d^{[\la]}.
$$
The maps $\Phi_{d}$ and $\Psi_d$ satisfy the double centralizer
property, i.e.
\begin{align*}
\Phi_{d}(\W)=&\End_{H_d^{[\la]}}(V^{\otimes d}),\\
&\End_{\W}(V^{\otimes d})^{\operatorname{op}} =
\Psi_d(H_d^{[\la]}).
\end{align*}
\end{theorem}

\begin{proof}[Sketch of a proof]
The first statement follows from a direct computation of the
minimal polynomial of $x_1$ by using Lemma~\ref{W:action}.

The commutativity of the two actions follow from the theorem of
Arakawa and Suzuki via the Skryabin equivalence.

The double centralizer property follows from the associated graded
version (i.e. the Vust duality in Section~\ref{W:subsec:vust}) and
a standard filtered algebra argument.
\end{proof}

\begin{remark}
The higher level Schur duality has applications to the BGG
category $\mc O$ for $\ga$; see \cite{BK3} for more details.
\end{remark}

\section{$W$-(super)algebras  in positive characteristic}
\label{W:sec:charp}

In this section, we will present the $W$-algebras and
$W$-superalgebras in positive characteristic, following \cite{Pr1,
WZ1, WZ2}.

\subsection{Restricted Lie superalgebras}
\label{W:subsec:superalg}

Let $F$ be an algebraically closed field of characteristic $p>2$.
A Lie superalgebra $\ga = \ev{\ga} \oplus \od{\ga}$ over $F$ is  a
{\em {restricted} Lie superalgebra}, if there is a $p$th power map
$\ev \ga \to \ev \ga, x \mapsto \pth{x},$ such that the even
subalgebra $\ev \ga$ is a restricted Lie algebra and the odd part
$\od \ga$ is a restricted module by the adjoint action of $\ev
\ga$. The notion of restricted Lie algebras was introduced by
Jacobson, cf. Jantzen \cite{Ja1} for an excellent review of
modular representations of Lie algebras.

\begin{example} \label{W:pmap}
The $p$-power map on the general linear Lie algebra $\gl_n$ is
given by $x \mapsto \overbrace{xx\cdots x}^p$, the $p$th product
for the underlying associative algebra.

The general linear Lie superalgebra $\gl_{m|n}$, which consists of
$(m+n) \times (m+n)$ matrices and whose even subalgebra is
isomorphic to $\gl_m \oplus \gl_n$, is a restricted Lie
superalgebra.
\end{example}

All the Lie (super)algebras considered here will be restricted. If
one is only interested in Lie algebras, one can always set $\od
\ga =0$ below.

We list all the basic classical Lie superalgebras over $F$ with
restrictions on $p$ as follows (the general linear Lie
superalgebra, though not simple, is also included).

\vspace{.4cm}

\begin{center}
\begin{tabular}{|c|c|}
\hline
Lie superalgebra & Characteristic of $K$\\
\hline $\gl_{m|n}$ & $p > 2$\\
\hline $\mathfrak{sl}_{m|n}$ & $p > 2, p\nmid (m-n)$\\
\hline $B(m,n), C(n), D(m,n)$ & $p > 2$ \\
\hline $D(2,1;  \alpha)$ & $p > 3$ \\
\hline $F(4)$ & $p >15$ \\
\hline $G(3)$ & $p >15$ \\
\hline
\end{tabular}
\vspace{.1cm}
\end{center}

The queer Lie superalgebra $\qn$ is a Lie subalgebra of the
general linear Lie superalgebra $\gl_{n|n}$, which consists of
matrices of the form:
\begin{equation*}  \label{W:q(n)}
\begin{pmatrix}
A&B\\
B&A\\
\end{pmatrix},
\end{equation*}
where $A$ and $B$ are arbitrary $n\times n$ matrices. Note that
the even subalgebra $\ev {(\qn)}$ is isomorphic to $\gl_n$ and the
odd part $\od {(\qn)}$ is another isomorphic copy of $\gl_n$ under
the adjoint action of $\ev {(\qn)}$.

\subsection{Reduced enveloping superalgebras}
\label{W:subsec:reduced}

Let $\ga$ be a restricted Lie superalgebra. For each $x \in
\ev{\ga}$, the element $x^p - \pth x \in U(\ga)$ is central by
definition of the $p$th power map. We refer to $\Zp(\ga) =
K\langle x^p - \pth x \, |\,x \in \ev \ga \rangle$ as the
$p$-center of $U(\ga)$. Let $x_1, \ldots, x_s$ (resp. $y_1,
\ldots, y_t$) be a basis of $\ev \ga$ (resp. $\od \ga$). As a
consequence of the PBW theorem for $U(\ga)$, $\Zp(\ga)$ is a
polynomial algebra isomorphic to $K[x_i^p - x_i^{[p]} \vert \;
i=1,\ldots,s]$, and the enveloping superalgebra $U(\ga)$ is free
over $\Zp(\ga)$ with basis
\[
\{x_1^{a_1} \cdots x_s^{a_s} y_1^{b_1} \cdots y_t^{b_t} \,|\, 0
\leq a_i < p; \, b_j = 0, 1 \text{ for all }i, j \}.
\]

\begin{xca} \label{W:xca:free}
Prove that the $p$-center $\Zp(\ga)$ is a polynomial algebra
isomorphic to $K[x_i^p - x_i^{[p]} \vert \; i=1,\ldots,s]$, and
that  $U(\ga)$ is free over $\Zp(\ga)$ with basis
\[
\{x_1^{a_1} \cdots x_s^{a_s} y_1^{b_1} \cdots y_t^{b_t} \,|\, 0
\leq a_i < p; \, b_j = 0, 1 \text{ for all }i, j \}.
\]
\end{xca}

Let $V$ be a simple $U(\ga)$-module. By Schur's lemma, the central
element $x^p - \pth x$ for $x \in \ev \ga$ acts by a scalar
$\zeta(x)$, which can be written as $\chi_V(x)^p$ for some $\chi_V
\in \ev \ga^*$. We call $\chi_V$ the {\em $p$-character} of the
module $V$. We often regard $\chi \in \ga^*$ by letting $\chi (\od
\ga) =0$.

Fix $\chi \in \ev \ga^*$. Let $J_{\chi}$ be the ideal of $U(\ga)$
generated by the even central elements $x^p - \pth x - \chi(x)^p$.
The quotient algebra $U_{\chi}(\ga) := U(\ga)/J_{\chi}$ is called
{\em the reduced enveloping superalgebra} with $p$-character
$\chi$. It follows from Exercise~\ref{W:xca:free} that the
superalgebra $U_{\chi}(\ga)$ has a linear basis
\[
\{x_1^{a_1} \cdots x_s^{a_s} y_1^{b_1} \cdots y_t^{b_t} \,|\, 0
\leq a_i < p; \, b_j = 0, 1 \, \text{for all }i, j \}.
\]
In particular, $\dim U_{\chi}(\ga) = p^{\dim \ev\ga} 2^{\dim
\od\ga}$.

A superalgebra analogue of Schur's Lemma states that the
endomorphism ring of an irreducible module of an associative
superalgebra is either one-dimensional or two-dimensional (in the
latter case it is isomorphic to a Clifford algebra), cf. e.g.
Kleshchev \cite[Chap.~12]{Kle}. An irreducible module is {\em of
type $\texttt M$} if its endomorphism ring is one-dimensional and
it is {\em of type $\texttt Q$} otherwise.

\subsection{The good $\Z$-gradings}
\label{W:subsec:supergood}

Let $\ga$ be one of the basic classical Lie superalgebras. The Lie
superalgebra $\ga$ admits a nondegenerate invariant even bilinear
form $(\cdot | \cdot)$, whose restriction on $\ev \ga$ gives an
isomorphism $\ga_{\bar{0}} \xrightarrow{\sim} \ga_{\bar{0}}^*$.
Let $\chi \in \ga_{\bar{0}}^*$ be a nilpotent element, that is, it
is the image of some nilpotent element $e\in \ga_{\bar{0}}$ under
the above isomorphism. Then $\ga_{\chi} :=\{x\in \ga \mid \chi
([x,y])=0, \forall y \in \ga\}$ is equal to the usual centralizer
$\ga_e$. We shall denote $\sudim \ga = \dim \ev\ga | \dim \od\ga$.

On the other hand, $\qn$ admits a nondegenerate invariant odd
bilinear form $(\cdot | \cdot)$, which induces an isomorphism
${(\qn)}_{\bar{1}} \xrightarrow{\sim} {(\qn)}_{\bar{0}}^*$. Hence
a nilpotent $p$-character $\chi$ corresponds to an odd nilpotent
element $e \in {(\qn)}_{\bar{1}}$.

\begin{proposition}\label{W:th:grading}
Let $\ga$ be a basic classical Lie superalgebra or a queer Lie
superalgebra. Then there exists a good $\Z$-grading $\ga =
\oplus_{j \in \Z}\ga_j$ as defined in \ref{W:Dynkin gr}. This
$\Z$-grading is compatible with the $\Z_2$-grading, that is,
$\ga_j =\ga_{j,\bar 0} \oplus \ga_{j,\bar 1}$ where $\ga_{j,a} =
\ga_j \cap \ga_a$ for $a \in \Z_2, j \in \Z$. In particular, we
have a refinement of \eqref{W:numbersimple}:
\begin{equation}\label{W:grading5}
\sudim \ga_e = \sudim \ga_0 + \sudim \ga_1 .
\end{equation}
\end{proposition}

\begin{proof}[Sketch of a proof]
For Lie superalgebras of type $A,B,C,D$ and type $Q$, we use the
classification of nilpotent orbits for classical Lie algebras
which leads to an elementary way of defining $\Z$-gradings, cf.
Jantzen \cite[Chapter~1]{Ja2}. For the exceptional Lie superalgebras, the
assumptions on $p$ in the Table of \ref{W:subsec:superalg} ensure
the existence of an $\sgl_2$-triple $\{e,h,f\}$ and the
semisimplicity of the adjoint action on $\ga$ by the
$\sgl_2$-triple (see Carter \cite{Ca}). The action of $F h$ on
$\ga$ lifts to a torus action which provides the desired grading
on $\ga$.
\end{proof}

\subsection{The subalgebra $\mathfrak m$}  \label{W:sec:subalgm}

Let $\ga$ be one of the basic classical Lie superalgebras or $\qn$
with a good $\Z$-grading $\ga = \oplus_{j \in \Z}\ga_j$.  Define a
bilinear form $\langle \cdot, \cdot\rangle$ on $\ga_{-1}$ by
\begin{equation}\label{W:ssform}
\langle x, y \rangle := \chi([x, y]), \quad x,y \in \ga_{-1}.
\end{equation}

The following is a super analogue of Lemma~\ref{W:skewform}.
\begin{xca}  \label{W:xca:skewSUSY}
Show that \eqref{W:ssform} defines an even, non-degenerate,
skew-supersymmetric form $\langle\cdot,\cdot\rangle$ on
$\ga_{-1}$, that is, a non-degenerate bilinear form such that
$\langle \ga_{-1,\bar 0},\ga_{-1,\bar 1}\rangle =0$, and that the
restriction of the form to $\ga_{-1,\bar 0}$ (resp. $\ga_{-1,\bar
1}$) is skew-symmetric (resp. symmetric).
\end{xca}

It follows that $\dim \ga_{-1, \bar 0}$ is even. Take $\lag_{\bar
0} \subseteq \ga_{-1, \bar 0}$ to be a maximal isotropic subspace
with respect to $\langle\cdot, \cdot\rangle$. It satisfies $\dim
\lag_{\bar 0} = \hf \dim \ga_{-1, \bar 0}$.

Denote $r = \dim \ga_{-1, \bar 1}$. There is a basis $v_1,\ldots,
v_r$ of $\ga_{-1, \bar 1}$ under which the symmetric form
$\langle\cdot, \cdot\rangle$ has matrix form $E_{1r}+E_{2,r-1}
+\ldots +E_{r1}$. If $r$ is even, take $\lag_{\bar 1} \subseteq
\ga_{-1, \bar 1}$ to be the subspace spanned by $v_1, \ldots,
v_{\frac{r}{2}}$. If $r$ is odd, take $\lag_{\bar 1} \subseteq
\ga_{-1, \bar 1}$ to be the subspace spanned by $v_1, \ldots,
v_{\frac{r-1}{2}}$. Set $\lag=\lag_{\bar 0} \oplus \lag_{\bar 1}$
and introduce the subalgebras
\begin{align*}
\m &= \lag  \bigoplus \bigoplus_{j \geq 2} \ga_{-j},\\
\m' &=\begin{cases} \m \oplus F v_{\frac{r+1}{2}}, & \text{for $r$
odd}\\
\m, & \text{for $r$ even}.
\end{cases}
\end{align*}

The subalgebra $\m$ is $p$-nilpotent, and the linear function
$\chi$ vanishes on the $p$-closure of $[\m, \m]$. It follows that
$U_{\chi}(\m)$ has the trivial module, denoted by $F_\chi$, as its
only irreducible module and $U_{\chi}(\m)/N_{\m} \cong F_\chi$
where $N_{\m}$ is the Jacobson radical of $U_{\chi}(\m)$.

\begin{proposition}\label{W:prop:freeness}
Let $\ga$ be a basic classical Lie superalgebra or $\qn$. Then
every $U_{\chi}(\ga)$-module is $U_{\chi}(\m)$-free.
\end{proposition}
Proposition~\ref{W:prop:freeness} in the case when $\ga$ is the
Lie algebra of a reductive algebraic group was first established
\cite{Pr1} using the machinery of support varieties (cf. \cite{FP,
Ja1}) in an essential way. The proof of
Proposition~\ref{W:prop:freeness} in \cite{WZ1} follows an idea of
Skryabin \cite{Sk2}, completely bypassing the machinery of support
varieties (which has yet to be developed for modular Lie
superalgebras).

\subsection{The super Kac-Weisfeiler  conjecture}
\label{W:subsec:superKW}

The following theorem in the present form is due to \cite{WZ1,
WZ2}, which generalize the original Kac-Weisfeiler conjecture
\cite{WK} (theorem of Premet \cite{Pr1}) when $\ga$ is a Lie
algebra of a reductive groups. Moreover, based on a deformation
approach of Premet and Skryabin, a new proof \cite{Zh} has
optimally improved the assumption on $p$ (that is, one may assume
$p>3$ for the last two exceptional superalgebras in the Table of
\ref{W:subsec:superalg}).

\begin{theorem}\label{W:th:KW-N}
Let $\ga$ be a basic classcial Lie superalgebra or $\qn$, and let
$\chi \in \ev \ga^*$ be nilpotent. Let $d_i = \dim \ga_i - \dim
(\ga_{\chi})_i,\; i \in \Z_2$. Then the dimension of every
$U_{\chi}(\ga)$-module $M$ is divisible by $p^{\frac{d_0}{2}}
2^{\lfloor \frac{d_1}{2}\rfloor}$, where $\lfloor
\frac{d_1}{2}\rfloor$ denotes the least integer which is $\geq
\frac{d_1}{2}$.
\end{theorem}

\begin{proof}
By (\ref{W:grading5}) and since $\sudim \ga_j = \sudim \ga_{-j}$, we
have
\begin{equation}\label{W:m-dim}
\sudim \ga - \sudim \ga_{\chi} = 2\sum_{j \geq 2} \sudim \ga_{-j}
+ \sudim \ga_{-1}.
\end{equation}
In particular, $r :=\dim \ga_{-1, \bar 1}$ and $d_1$ have the same
parity. It follows now from the definition of $\m$ that either (1)
$\frac{d_0}{2} | \frac{d_1}{2} = \sudim \m$ when $d_1$ is even, or
(2) $\frac{d_0}{2} | \frac{d_1 - 1}{2} = \sudim \m$ when $d_1$ is
odd.

In case (1), the theorem follows immediately from
Proposition~\ref{W:prop:freeness}. Note that $d_1 =d_0$ and hence
the case (2) never occurs for $\ga =\qn$.

In case (2), the induced $U_{\chi}(\m')$-module $V = U_{\chi}(\m')
\otimes_{U_{\chi}(\m)} F$ is two-dimensional, irreducible, and
admits an odd automorphism of order $2$ induced from
$v_{\frac{r+1}{2}}$. By Frobenius reciprocity, it is the only
irreducible $U_{\chi}(\m')$-module. Thus, $U_{\chi}(\m')/ N_{\m'}$
is isomorphic to the simple associative superalgebra
$\mathfrak{q}_1^{\text{as}}$ (which coincides with $\mf q_1$ as a
vector superspace), and the unique simple $U_{\chi}(\m')$-module
$V$ is of type $\texttt{Q}$.

Hence, for each $U_{\chi}(\ga)$-module $M$, the subspace $M^{\m}$
of $\m$-invariants in $M$ , which coincides with $M^{\n}$, is a
module over the superalgebra $U_{\chi}(\m')/ N_{\m} \cong
\mathfrak{q}_1^{\text{as}}$. Since the (unique) simple module of
$\mathfrak{q}_1^{\text{as}}$ is two-dimensional, $\dim M^{\m}$ is
divisible by $2$. Now the isomorphism $M \cong U_{\chi}(\m)^*
\otimes M^{\m}$ implies the desired divisibility.
\end{proof}

\subsection{The modular $W$-superalgebras}
\label{W:subsec:modularW}

Denote by $\mathcal{Q}_{\chi}$ the induced $U_{\chi}(\ga)$-module
$U_{\chi}(\ga)\otimes_{U_{\chi}(\m)}F_{\chi}$.

\begin{definition}   \label{W:def:Wsuperalgebra}
The {\em (modular) $W$-superalgebra} associated to the pair $(\ga,
\chi)$ is the following associative $F$-superalgebra
$$
W_\chi =\text{End}_{U_{\chi}(\ga)}(\mathcal{Q}_{\chi})^{\op}.
$$
\end{definition}

Theorem~\ref{W:th:KW-N} can be somewhat strengthened in the
following form \cite{WZ1, WZ2}, which is a superalgebra
generalization of a theorem of Premet \cite[Theorem~2.3]{Pr2}.
\begin{theorem} \label{W:th:finW}
Set $\delta = \dim U_{\chi}(\m)$. Then $\mathcal{Q}_{\m}$ is a
projective $U_{\chi}(\ga)$-module and
\[
U_{\chi}(\ga) \cong M_{\delta} (W_\chi).
\]
\end{theorem}

\begin{proof}
Let $V_1, \ldots, V_s$ (resp. $W_1, \ldots, W_t$) be all
inequivalent simple $U_{\chi}(\ga)$-modules of type $\texttt{M}$
(resp. of type $\texttt{Q}$). Let $P_i$ (resp. $Q_j$) denote the
projective cover of $V_i$ (resp. $W_j$). By
Proposition~\ref{W:prop:freeness}, $V_i$ and $W_j$ are free over
$U_{\chi}(\m)$. It follows by Frobenius reciprocity that
\[
\dim \Hom_{\ga}(\mathcal{Q}_{\chi}, V_i)= \dim \Hom_{\m}(F_{\chi},
V_i) = :a_i.
\]
By Frobenius reciprocity,
\[
\dim \Hom_{\ga}(\mathcal{Q}_{\chi}, W_j)= \dim \Hom_{\m}(F_{\chi},
W_j) = \dim \Hom_{\m}(U_{\chi}(\m), W_j)
\]
which has to be an even number, say $2b_j$, since as a type
$\texttt{Q}$ module $W_j$ admits an odd involution commuting with
$\m$. It follows that the ranks of the free $U_{\chi}(\m)$-modules
$V_i$ and $W_j$ are $a_i$ and $2b_j$ respectively. Put
\[
P = (\bigoplus_{i=1}^{s}  P_i^{a_i}) \bigoplus (\bigoplus_{j=1}^t
Q_j^{b_j}).
\]
Then $P$ is projective and has the same head as
$\mathcal{Q}_{\m}$. So there is a surjective homomorphism $\psi: P
\to \mathcal{Q}_{\chi}$.

Since $\dim V_i = \delta a_i$ and $\dim W_j = 2\delta b_j$, by
Wedderburn theorem for superalgebras (cf. Kleshchev
\cite[Theorem~12.2.9]{Kle}) the left regular
$U_{\chi}(\ga)$-module is isomorphic to $P^{\delta}$. The equality
of dimensions
\[
\dim P = \dim U_{\chi}(\ga)/{\delta} = \dim \mathcal{Q}_{\m}
\]
implies that $\psi$ is an isomorphism. Finally,
\begin{eqnarray*}
U_{\chi}(\ga) &\cong&
\text{End}_{U_{\chi}(\ga)}(U_{\chi}(\ga))^{\op} \cong
\text{End}_{U_{\chi}(\ga)}(P^{\delta})^{\op}
\\
&\cong& (M_{\delta} (\text{End}_{U_{\chi}(\ga)}(P)))^{\op} \cong
M_{\delta} (W_\chi).
\end{eqnarray*}
This completes the proof of the theorem.
\end{proof}

\begin{remark}   \label{W:rem:Wsuper}
In light of Gan-Ginzburg's definition of $W$-algebras over $\C$,
it makes better sense to define a (modular) $W$-superalgebra over
$F$ as
$$
W_\chi' := (U_\chi(\ga)/\mf I_\chi)^{\ad \n} \equiv \{\bar{y} \in
U_\chi (\ga)/\mf I_\chi \mid [a, y] \in \mf I_{\chi}, \forall a
\in \n\},
$$
where $\mf I_\chi$ denotes the left ideal of $U_\chi(\ga)$
generated by $N_\m$, with multiplication $\bar{y}_1 \bar{y}_2 =
\overline{y_1y_2},$ for $\bar{y}_1, \bar{y}_2 \in \mc W_{\chi}'$.
In the case when $d_1$ (or equivalently, $\dim \ga_{-1, \bar 1}$)
is odd, $ W_\chi'$ is in general a subalgebra of $W_\chi \cong
(U_\chi(\ga)/\mf I_\chi)^{\ad \m}$.

Similar subtle differences appear in various constructions of
$W$-superalgebras over $\C$.
\end{remark}

\section{Further work and open problems}
\label{W:sec:open}

\subsection{Representations of $\W$ of type $A$}
\label{W:subsec:A}

The representation theory of finite $W$-algebras has been most
adequately developed for $\ga =\gl_N$.  Brundan and Kleshchev
\cite{BK1} showed that the finite $W$-algebras were quotient
algebras of what they introduced and called shifted Yangians,
associated to a given pyramid $\pi$. In the special case when
$\pi$ is of rectangular shape, such a connection between finite
$W$-algebras and the truncated Yangians (which are quotients of
the Yangians introduced by Drinfeld) was found by Ragoucy and
Sorba \cite{RS}, and the truncated Yangians were studied by
Cherednik \cite{Ch}. Also see Brown \cite{Br} for a generalization
of \cite{RS} to other classical Lie algebras.

Such a connection with shifted Yangians was subsequently explored
systematically in \cite{BK2}, where the finite-dimensional
$\W$-modules are classified. Moreover, a highest weight
representation theory of $\W$ is developed, and a Kazhdan-Lusztig
solution to the finite-dimensional irreducible character problem
of $\W$ is obtained. A formulation of Kazhdan-Lusztig conjectures
for general finite $W$-algebras was given by de Vos and van Driel
\cite{VD}.

\begin{remark} \label{W:affineW}
The affine Yangian associated to the affine algebra of $\ga$ have
been studied by Guay \cite{Gu}. We remark that the affine Yangian
for $\gl_1$ should be identified with the $\mc W_{1+\infty}$
algebra. The main result of \cite{FKRW} can be reformulated as
follows:
{\em The vertex algebra for the truncated affine Yangian of
$\gl_1$ at level $N$ is isomorphic to the affine $W$-algebra
associated to $\gl_N$ and a regular nilpotent element at level
$N$.}
\end{remark}

\subsection{Finite-dimensional $\W$-modules and primitive ideals}
\label{W:subsec:finite}

The importance of finite-dimensional $\W$-modules was realized by
Premet \cite{Pr3} in connection to primitive ideals for $U(\ga)$,
where he showed that the corresponding $\ga$-modules under the
Skryabin equivalence have the closure $\overline{\mc O}_e$ of a
nilpotent orbit as their associated varieties and $\hf \dim
\overline{\mc O}_e$ as their Gelfand-Kirillov dimensions. Further
different approaches toward connections between finite-dimensional
$\W$-modules and primitive ideals are developed in \cite{Gi, Lo1,
Lo2, Pr4, Pr5} (see \cite[Conjecture~3.2]{Pr3} for motivation).

Among all finite-dimensional $\W$-modules, the one-dimensional
modules stand out (cf. \cite[Conjecture~3.1]{Pr3}), and they are
studied intensively in \cite{Lo1, Lo4, Pr5, GRU}.

The upshot of all these developments is that finite-dimensional
modules are shown to exist for every $\W$, and one-dimensional
modules are shown to exist (modulo some open cases in $E_8$). The
existence of one-dimensional $W$-modules implies that the
dimensional bound provided by Kac-Weisfeiler conjecture is optimal
for characteristic $p \gg 0$ \cite{Pr5}.

\subsection{Category $\mc O$}
\label{W:subsec:catO}

The category $\mc O$ for finite $W$-algebras was well studied in
\cite{BK2} in type $A$. One needs to overcome substantial
technical difficulties to formulate the category $\mc O$ for
finite $W$-algebras in other types, see \cite{BGK, Lo3, Go, Pr3,
We} for research in this direction.

\subsection{Open problems}
\label{W:subsec:problems}

As in any reasonable representation theory, one of the basic
problems in $W$-algebras is the following.

\begin{problem}   \label{W:prob:KL}
Find a (Kazhdan-Lusztig type) solution to the irreducible
character problem for category $\mc O$ of $\W$-modules.
\end{problem}

As we have seen that finite $W$-superalgebras are natural subjects
of study in connections to the natural yet nontrivial superalebra
generalization of the Kac-Weisfeiler conjecture \cite{WZ1, WZ2}.
We remark that the queer Lie superalgebra $\qn$ is often referred
to as the truly superanalogue of $\gl_n$.

\begin{problem}   \label{W:prob:super}
Develop systematically the structures and representation theory of
finite $W$-superalgebras over $\C$, for both basic classical Lie
superalgebras and the queer Lie superalgebras.
\end{problem}
Various aspects of finite $W$-superalgebras over $\C$ are
currently being developed in the Virginia dissertation of
Yun-Ning~Peng.

\begin{problem}   \label{W:prob:quantum}
Formulate the notion of quantum finite $W$-algebras associated to
quantum groups $U_q(\ga)$ and then develop its representation
theory (for both $q$ being generic and $q$ being a root of unity).
\end{problem}
In the case of regular nilpotent elements, there has been a very
interesting work of Sevostyanov \cite{Sev}.

It is shown \cite{PPY} that the algebra of invariants of the
centralizer $\ga_e$ in the symmetric algebra of $\ga_e$ is often a
polynomial algebra in $\text{rank} (\ga)$ generators, even though
$\ga_e$ is in general non-reductive; also see \cite{BrB} for a
constructive approach in type $A$. One should compare with the
fact that the center of $\W$ is always isomorphic to $\mc Z(\ga)$,
which is independent of $\chi$ (see \cite{BK3} for type $A$ and
\cite[Footnote~2]{Pr3} in general).

We pose the problem of finding a skew analogue of the above
results.
\begin{problem}  \label{W:prob:skew}
When is the algebra of $\ga_e$-invariants in the exterior algebra
of $\ga_e$ an exterior algebra in $\text{rank} (\ga)$ generators?
\end{problem}
Note that the answer is affirmative when $e =0$ or $e$ is regular
nilpotent by classical results. A counterpart of this problem for
$\W$ would naturally lead to the following.

\begin{problem}   \label{W:prob:cohomW}
Develop a cohomology theory of finite $W$-algebras.
\end{problem}

The connection between shifted Yangians and finite $W$-algebras of
type $A$ is provided by explicit formulas \cite{BK1}. Being
explicit, these formulas greatly facilitate the detailed study in
type $A$, but on the other hand, they seem to be too rigid to
allow direct generalizations to other types or to the affine
setting. This leads to the following.

\begin{problem}   \label{W:prob:affineW}
Redevelop the connections between Yangians and finite $W$-algebras
in type $A$ from the BRST approach.
Formulate a generalization of Remark~\ref{W:affineW}, which should
be viewed as an affine analogue of the above connections in type
$A$.
\end{problem}

\end{document}